\documentclass[english]{article}
\usepackage{lmodern}
\usepackage[T1]{fontenc}
\usepackage[latin9]{inputenc}
\usepackage[a4paper]{geometry}
\usepackage{color}
\usepackage{babel}
\usepackage{textcomp}
\usepackage{mathrsfs}
\usepackage{enumitem}
\usepackage{amsmath}
\usepackage{amsthm}
\usepackage{amssymb}
\usepackage{stmaryrd}
\usepackage{graphicx}
\usepackage[all]{xy}
\PassOptionsToPackage{normalem}{ulem}
\usepackage{ulem}
\usepackage[unicode=true,pdfusetitle,
 bookmarks=true,bookmarksnumbered=false,bookmarksopen=false,
 breaklinks=false,pdfborder={0 0 0},pdfborderstyle={},backref=false,colorlinks=true]
 {hyperref}
\hypersetup{
 linkcolor=blue,  citecolor=blue, urlcolor=blue}

\makeatletter

\providecommand{\tabularnewline}{\\}

\numberwithin{equation}{section}
\numberwithin{figure}{section}
\theoremstyle{plain}
\newtheorem{thm}{\protect\theoremname}[section]
\theoremstyle{plain}
\newtheorem{conjecture}[thm]{\protect\conjecturename}
\theoremstyle{remark}
\newtheorem*{acknowledgement*}{\protect\acknowledgementname}
\theoremstyle{definition}
\newtheorem{defn}[thm]{\protect\definitionname}
\theoremstyle{plain}
\newtheorem{fact}[thm]{\protect\factname}
\theoremstyle{remark}
\newtheorem{rem}[thm]{\protect\remarkname}
\theoremstyle{plain}
\newtheorem{lem}[thm]{\protect\lemmaname}
\theoremstyle{plain}
\newtheorem{prop}[thm]{\protect\propositionname}
\theoremstyle{remark}
\newtheorem{claim}[thm]{\protect\claimname}
\theoremstyle{plain}
\newtheorem{cor}[thm]{\protect\corollaryname}
\theoremstyle{definition}
\newtheorem{problem}[thm]{\protect\problemname}
\theoremstyle{definition}
\newtheorem{example}[thm]{\protect\examplename}
\theoremstyle{plain}
\newtheorem*{thm*}{\protect\theoremname}
\newlist{casenv}{enumerate}{4}
\setlist[casenv]{leftmargin=*,align=left,widest={iiii}}
\setlist[casenv,1]{label={{\itshape\ \casename} \arabic*.},ref=\arabic*}
\setlist[casenv,2]{label={{\itshape\ \casename} \roman*.},ref=\roman*}
\setlist[casenv,3]{label={{\itshape\ \casename\ \alph*.}},ref=\alph*}
\setlist[casenv,4]{label={{\itshape\ \casename} \arabic*.},ref=\arabic*}

\usepackage[abbrev,nobysame,alphabetic]{amsrefs}
\def\MR#1{}
\date{}
\usepackage{paralist}

  \pltopsep=\medskipamount
  \setdefaultleftmargin{2em}{2em}{1.87em}{1.7em}{1em}{1em}
\setlist[casenv]{itemsep=2pt,leftmargin=*,align=left,itemindent=3em}

\theoremstyle{plain}
\newtheorem{claim}[thm]{\protect\claimname}

\theoremstyle{definition}
\newtheorem{rem}[thm]{Remark}

\DeclareMathOperator{\im}{Im}
\DeclareMathOperator{\len}{len}

\makeatother

\providecommand{\acknowledgementname}{Acknowledgement}
\providecommand{\casename}{Case}
\providecommand{\claimname}{Claim}
\providecommand{\conjecturename}{Conjecture}
\providecommand{\corollaryname}{Corollary}
\providecommand{\definitionname}{Definition}
\providecommand{\examplename}{Example}
\providecommand{\factname}{Fact}
\providecommand{\lemmaname}{Lemma}
\providecommand{\problemname}{Problem}
\providecommand{\propositionname}{Proposition}
\providecommand{\remarkname}{Remark}
\providecommand{\theoremname}{Theorem}

\begin{document}
\global\long\def\pplus{\overset{{\scriptscriptstyle \perp}}{\oplus}}%
\global\long\def\oton#1#2{#1_{1},\dots,#1_{#2}}%
\global\long\def\rk{\text{rk}}%
\global\long\def\tr{\text{tr}}%
\global\long\def\adj{\text{adj}}%
\global\long\def\nin{\notin}%
\global\long\def\id{\text{Id}}%
\global\long\def\makbil{\text{parallel}}%
\global\long\def\meridian{\text{meridian}}%
\global\long\def\sgn{\text{Sgn}}%
\global\long\def\fix{\text{Fix}}%
\global\long\def\hess{\text{Hess}}%
\global\long\def\ooton#1#2{#1_{0},\dots,#1_{#2}}%
\global\long\def\otherwise{\text{Otherwise}}%
\global\long\def\foton#1#2#3{#1\left(#2_{1}\right),\dots,#1\left(#2_{#3}\right)}%
\global\long\def\norm#1{\left\Vert #1\right\Vert }%
\global\long\def\nil{\text{Nill}}%
\global\long\def\fix{\text{Fix}}%
\global\long\def\spec{\text{Spec}}%
\global\long\def\ind{\text{Ind}}%
 
\global\long\def\hom{\mathrm{Hom}}%
\global\long\def\ob{\text{Ob}}%
\global\long\def\coker{\text{coker}}%
\global\long\def\rad{\text{Rad}}%
\global\long\def\supp{\text{Supp}}%
\global\long\def\aut{\text{Aut}}%
\global\long\def\gal{\text{Gal}}%
\global\long\def\ann{\text{Ann}}%
\global\long\def\mayervi{\text{Mayer-Vietoris}}%
\global\long\def\conv{\text{conv}}%
\global\long\def\diam{\text{diam}}%
\global\long\def\length{\text{len}}%
\global\long\def\tp{\text{tp}}%
\global\long\def\lcm{\text{lcm}}%
\global\long\def\core{\text{Core}}%
\global\long\def\ad{\text{ad}}%
\global\long\def\ord{\text{ord}}%
\global\long\def\rank{\text{rank}}%
\global\long\def\uindex#1#2{\overset{{\scriptscriptstyle #2}}{#1}}%
\global\long\def\nuindex#1#2{\overset{{\scriptscriptstyle #2}}{#1}\,^{-1}}%
\global\long\def\involution{\overline{{\scriptstyle \bigbox}}}%
\global\long\def\mold{{\scriptstyle \varodot}}%
\global\long\def\type{\text{Type}}%

\title{On algebraic extensions and decomposition\\
of homomorphisms of free groups}
\author{Noam Kolodner}
\maketitle
\begin{abstract}
We give a counterexample to a conjecture by Miasnikov, Ventura and
Weil, stating that an extension of free groups is algebraic if and
only if the corresponding morphism of their core graphs is onto, for
every basis of the ambient group. In the course of the proof we present
a partition of the set of homomorphisms between free groups which
is of independent interest.
\end{abstract}

\section{Introduction}

The purpose of this paper is to develop new machinery for the study
of morphisms between core graphs associated with free groups (see
definitions in §\ref{sec:Preliminaries}). For this purpose, we construct
a category which enriches that of free groups and enables to study
surjectivity with respect to arbitrary bases. The main application
we give is a counterexample to a conjecture by Miasnikov, Ventura,
and Weil \cite{MR2395796} which was revised by Parzanchevski and
Puder \cite{MR3211804}.
\begin{conjecture}[\cite{MR2395796}]
\label{conj:old}Let $F_{Y}$ be a free group on the set $Y$, and
let $H\leq K\leq F_{Y}$ be subgroups. If the morphism between the
associated core-graphs $\Gamma_{X}\left(H\right)\to\Gamma_{X}\left(K\right)$
is surjective for every basis $X$ of $F_{Y}$, then $K$ is an algebraic
extension of $H$.
\end{conjecture}

The converse is indeed true: if $H\leq K$ is an algebraic extension,
then $\Gamma_{B}(H)\rightarrow\Gamma_{B}\left(K\right)$ is indeed
onto for any basis $B$. In \cite{MR3211804}, Parzanchevski and Puder
constructed a counterexample for the conjecture in $F_{2}$, the free
group on two generators. Their proof relies heavily on idiosyncrasies
of the automorphism group of $F_{2}$, and they conjectured that these
are the only obstructions, revising the conjecture in two ways:
\begin{conjecture}[\cite{MR3211804}]
\label{conj:main}Let $H\leq K\leq F_{Y}$. If for every free extension
$F'$ of $F_{Y}$ and for every basis $X$ of $F'$ the morphism $\Gamma_{X}\left(H\right)\to\Gamma_{X}\left(K\right)$
is onto then $K$ is algebraic over $H$.
\end{conjecture}

\begin{conjecture}[\cite{MR3211804}]
\label{conj:Xge3}The original conjecture holds for $H\leq K\leq F_{Y}$
such that\ $\left|Y\right|>2$.
\end{conjecture}

In this paper, we show that both revised forms of the conjecture are
false. In fact, we prove a stronger statement:
\begin{thm}
\label{thm:main}The extension $\langle bbaba^{-1}\rangle<\langle b,aba^{-1}\rangle$
is not algebraic, but for every morphism $\varphi\colon F_{\{a,b\}}\rightarrow F_{X}$
with $X$ arbitrary and $\varphi(a),\varphi(b)\neq1$, the graph morphism
$\Gamma_{X}\left(\varphi\left(\langle bbaba^{-1}\rangle\right)\right)\rightarrow\Gamma_{X}\left(\varphi(\langle b,aba^{-1}\rangle)\right)$
is surjective.
\end{thm}

Our result naturally raises the question what is the algebraic interpretation
(if one exists at all) of the property of being ``onto in all bases'',
as the one suggested by the conjectures above turns out to be false.
In addition, the methods which we develop for the analysis of our
counterexample are of independent interest: The first idea is extending
the scope of inspection from the category of automorphisms to a larger
category of homomorphisms called FGR (``Free Groups with Restrictions''),
which is defined in §\ref{sec:The-category-FGR}. In the new category
we present a recursive decomposition of morphisms, which in turn enables
ones to verify whether the desired property of graphs holds for all
FGR homomorphisms. A second point of interest is the observation of
a property of the word $bbaba^{-1}$ which ensures that the recursive
process ends. This property is defined in §\ref{sec:TBD-finiteness},
and leads to further questions and applications, for example to solving
some forms of equations in free groups.
\begin{acknowledgement*}
I am grateful to Ori Parzanchevski for his guidance and for carefully
reading and commenting on a draft of this paper. I would also like
to thank the referee for the helpful comments. This research was supported
by ISF grant 1031/17 of Ori Parzanchevski.
\end{acknowledgement*}

\section{\label{sec:Preliminaries}Preliminaries}

We present the definitions we use in the paper, merging the language
of \cite{MR1882114} and \cite{MR695906}.
\begin{defn}[Graphs]
We use graphs in the sense of Serre \cite{MR0476875}: A \emph{graph}
$\Gamma$ is a set $V\left(\Gamma\right)$ of vertices and a set $E\left(\Gamma\right)$
of edges with a function $\iota\colon E\left(\Gamma\right)\to V\left(\Gamma\right)$
called the initial vertex map and an involution $\involution\colon E\left(\Gamma\right)\to E\left(\Gamma\right)$
with $\overline{e}\neq e$ and $\overline{\overline{e}}=e$. A \emph{graph
morphism} $f\colon\Gamma\to\Delta$ is a pair of set functions $f^{E}\colon E\left(\Gamma\right)\to E\left(\Delta\right)$
and $f^{V}\colon V\left(\Gamma\right)\to V\left(\Delta\right)$ that
commute with the structure functions. A \emph{path }in $\Gamma$ is
a finite sequence $\oton en\in E\left(\Gamma\right)$ with $\iota\left(\overline{e}_{k}\right)=\iota\left(e_{k+1}\right)$
for every $1\leq k<n$. The path is \emph{closed} if $\iota\left(\overline{e}_{n}\right)=\iota\left(e_{1}\right)$,
and \emph{reduced} if $e_{k+1}\neq\overline{e}_{k}$ for all $k$.
All the graphs in the paper are assumed to be connected unless specified
otherwise, namely, for every $e,f\in E\left(\Gamma\right)$ there
is a path $e,\oton en,f$.
\end{defn}

\begin{defn}[Labeled graphs]
Let $X$ be a set and let $X^{-1}$ be the set of its formal inverses.
We define $R_{X}$ to be the graph with $E\left(R_{X}\right)=X\cup X^{-1}$,
$V\left(R_{X}\right)=\left\{ *\right\} $, $\overline{x}=x^{-1}$
and $\iota\left(x\right)=*.$ An \emph{$X$-labeled graph }is a graph
$\Gamma$ together with a graph morphism $l\colon\Gamma\to R_{X}$.
A morphism of $X$-labeled graphs $\Gamma$ and $\Delta$ is a graph
morphism $f\colon\Gamma\to\Delta$ that commutes with the label functions.
Let $\mathcal{P}\left(\Gamma\right)$ be the set of all the paths
in $\Gamma$, let $F_{X}$ be the free group on $X$ and let $P=\oton en$
be a path. The edge part of the label function $l^{E}\colon E\left(\Gamma\right)\to E\left(R_{X}\right)$
can be extended to a function $l\colon\mathcal{P}\left(\Gamma\right)\to F_{X}$
by the rule $l\left(P\right)=l\left(e_{1}\right)\ldots\left(e_{n}\right)$.

A \emph{pointed }$X$-labeled graph is an $X$-labeled graph that
has a distinguished vertex called the base point. A morphism of a
pointed labeled graph sends the base point to the base point. This
constitutes a category called $X\text{-Grph}$. For a pointed $X$-labeled
graph $\Gamma$ we define $\pi_{1}\left(\Gamma\right)$ to be 
\[
\pi_{1}(\Gamma)=\left\{ l\left(P\right)\in F_{X}\,|\,P\text{ is a closed path beginning at the base point}\right\} .
\]
\end{defn}

\begin{defn}
Let $F_{X}$ be the free group on the set $X$. We define a category
$\text{Sub}\left(F_{X}\right)$, whose objects are subgroups $H\leq F_{X}$
and there is a unique morphism in $\hom\left(H,K\right)$ iff $H\leq K$.
It is easy to verify that $\pi_{1}$ is a functor from $X\text{-Grph}$
to $\text{Sub}\left(F_{X}\right)$.
\end{defn}

Note: The functor $\pi_{1}$ defined above is not the fundamental
group of $\Gamma$ as a topological space. Rather, if one views $\Gamma$
and $R_{X}$ as topological spaces and $l$ as a continuous function,
then $\pi_{1}$ is the image of the fundamental group of $\Gamma$
in that of $R_{X}$ under the group homomorphism induced by $l$.
\begin{defn}[Folding]
A labeled graph $\Gamma$ is \emph{folded }if $l\left(e\right)\neq l\left(f\right)$
holds for every two edges $e,f\in E\left(\Gamma\right)$ with $\iota\left(f\right)=\iota\left(e\right)$.
We notice that there is at most one morphism between two pointed folded
labeled graphs. If $\Gamma$ is not folded, there exist $e,f\in E\left(\Gamma\right)$
such that\ $\iota\left(e\right)=\iota\left(f\right)$ and $l\left(e\right)=l\left(f\right)$;
Let $\Gamma'$ be the graph obtained by identifying the vertex $\iota\left(\overline{e}\right)$
with $\iota\left(\overline{f}\right)$, the edges $e$ with $f$ and
$\overline{e}$ with $\overline{f}$. We say $\Gamma'$ is the result
of \emph{folding} $e$ and $f$. The label function $l$ factors through
$\Gamma'$, yielding a label function $l'$ on $\Gamma'$, and we
notice that $\pi_{1}\left(\Gamma\right)=\pi_{1}\left(\Gamma'\right)$.
\end{defn}

\begin{defn}[Core graph]
A \emph{core graph} $\Gamma$ is a labeled, folded, pointed graph
such that\ every edge in $\Gamma$ belongs in a closed reduced path
around the base point. Is case of finite graphs this is equivalent
to every $v\in V\left(\Gamma\right)$ having $\deg(v):=|\iota^{-1}(v)|>1$
except the base point which can have any degree.
\end{defn}

\begin{defn}
Let $X\text{-CGrph}$ be the category of connected, pointed, folded,
$X$-labeled core graphs. Define a functor $\Gamma_{X}\colon\text{Sub}\left(F_{X}\right)\to X\text{-CGrph}$
that associates to the subgroup $H\leq F_{X}$ a graph $\Gamma_{X}\left(H\right)$
(which is unique up to a unique isomorphism) such that\ $\pi_{1}(\Gamma_{X}(H))=H$.
\end{defn}

\begin{fact}[\cite{MR695906,MR1882114}]
The functors $\pi_{1}$ and $\Gamma$ define an equivalence between
the categories $X\text{-CGrph}$ and $\text{Sub}\left(F_{X}\right)$.
\end{fact}

The correspondence between the categories of $X\text{-CGrph}$ and
$\text{Sub}\left(F_{X}\right)$ follows from the theory of cover spaces.
Let us sketch a proof. We regard $R_{X}$ as a topological space and
look at the category of connected pointed cover spaces of $R_{X}$.
This category is equivalent to $\text{Sub}\left(F_{X}\right)$, following
from the fact that $R_{X}$ has a universal cover. Let $\Gamma$ be
a connected folded $X$-labeled core graph, viewed as a topological
space and $l$ as a continuous function. There is a unique way up
to cover isomorphism to extend $\Gamma$ to a cover of $R_{X}$. There
is also a unique way to associate a core graph to a cover space of
$R_{X}$. This gives us an equivalence between the category of connected
pointed cover spaces of $R_{X}$ and pointed connected folded $X$-labeled
core graphs.
\begin{defn}
By uniqueness of the core graph of a subgroup we can define a functor
$\core\colon\text{\ensuremath{X}-Grph}\to X\text{-CGrph}$ that associates
to a graph $\Gamma$ a core graph such that\ $\pi_{1}\left(\core\left(\Gamma\right)\right)=\pi_{1}\left(\Gamma\right)$.
\end{defn}

\begin{defn}
Let $\Gamma$ be a graph with a vertex $v$ of degree one which is
not the base-point. For $e=\iota^{-1}\left(v\right)$, let $\Gamma'$
be the graph with $V\left(\Gamma'\right)=V\left(\Gamma\right)\backslash\left\{ v\right\} $
and $E\left(\Gamma'\right)=E\left(\Gamma\right)\backslash\left\{ e,\overline{e}\right\} $.
We say that $\Gamma'$ is the result of \emph{trimming} $e$ from
$\Gamma$, and we notice that $\pi_{1}\left(\Gamma\right)=\pi_{1}\left(\Gamma'\right)$.
\end{defn}

\begin{rem}
For a finite graph $\Gamma$, after both trimming and folding $\left|E\left(\Gamma'\right)\right|<\left|E\left(\Gamma\right)\right|$.
If no foldings or trimmings are possible then $\Gamma$ is a core
graph. This means that after preforming a finite amount of trimmings
and foldings we arrive at $\core\left(\Gamma\right)$. It follows
from the uniqueness of $\core\left(\Gamma\right)$ that the order
in which one performs the trimmings and foldings does not matter.
\end{rem}

\begin{defn}[Algebraic extension]
Let $H\leq K\leq F_{X}$ be subgroups. The subgroup $K$ is said
to be an \emph{algebraic extension }of $H$ if there exists no proper
free factor $J$ of $K$ with $H\leq J<K$.
\end{defn}

Conjecture \ref{conj:old} is the converse to the following:
\begin{fact}[\cite{MR2395796}]
\label{fact:InversConj}If $H\leq K\leq F_{X}$ and $H\leq K$ is
an algebraic extension then $\Gamma_{Y}\left(H\right)\rightarrow\Gamma_{Y}\left(K\right)$
is onto for every basis $Y$ of $F_{X}$.
\end{fact}

Another motivation for this conjecture is the fact that in a specific
case it is true
\begin{fact}
\label{fact: CoverBasic}Let $Y$ be a finite set then $F_{Y}$ is
algebraic over $H$ iff $\Gamma_{X}\left(H\right)\to\Gamma_{X}\left(F_{Y}\right)$
is onto for every basis $X$ of $F_{Y}$ (To see why this is true
notice that $\Gamma_{X}\left(F_{Y}\right)$ is a bouquet of circles
for every $X)$.
\end{fact}

When we combine Facts \ref{fact:InversConj} and \ref{fact: CoverBasic}
we get a statement about intermediate subgroups which is another motivation
for the conjecture: Let $H\leq K\leq F_{Y}$ then $K$ is algebraic
over $H$ iff for every $L$ with $K\leq L\leq F_{Y}$ and every basis
$X$ of $L$ the morphism $\Gamma_{X}\left(H\right)\to\Gamma_{X}\left(K\right)$
is onto.

\section{\label{sec:The-category-FGR}The category FGR}
\begin{defn}[Whitehead graph]
A \emph{2-path} in a graph $\Gamma$ is a pair $\left(e,f\right)\in E\left(\Gamma\right)\times E\left(\Gamma\right)$
with $\iota\left(f\right)=\iota\left(\overline{e}\right)$ and $f\neq\overline{e}$.
If $\Gamma$ is $X$-labeled, the set 
\[
W\left(\Gamma\right)=\left\{ \left\{ l\left(e\right),l\left(\overline{f}\right)\right\} \mid\left(e,f\right)\text{ is a \ensuremath{2}-path in \ensuremath{\Gamma}}\right\} 
\]
forms the set of edges of a combinatorial (undirected) graph whose
vertices are $X\cup X^{-1}$, called the \emph{Whitehead graph }of
$\Gamma$. If $w\in F_{X}$ is a cyclically reduced word, the Whitehead
graph of $w$ as defined in \cite{MR1575455,MR1714852} and the Whitehead
graph of $\Gamma\left(\left\langle w\right\rangle \right)$ defined
here coincide. Let $W_{X}=W\left(R_{X}\right)$ be the set of edges
of the Whitehead graph of $R_{X}$, which we call the \emph{full Whitehead
graph}. Let $x,y\in X\cup X^{-1}$ and let $\left\{ x,y\right\} \in W_{X}$
be an edge. We denote $x.y=\left\{ x,y\right\} $ (this is similar
to the notation in \cite{MR1812024}).
\end{defn}

\begin{defn}
A homomorphism $\varphi\colon F_{Y}\to F_{X}$ is \emph{non-degenerate}
if $\varphi\left(y\right)\neq1$ for every $y\in Y$.
\end{defn}

\begin{defn}
Let $w\in F_{X}$ be a reduced word of length $n$. Define $\Gamma_{w}$
to be the $X$-labeled graph with $V\left(\Gamma_{w}\right)=\left\{ 1,\dots,n+1\right\} $
forming a path $P$ labeled by $l\left(P\right)=w$. Notice that $\Gamma_{w}\cong\Gamma_{w^{-1}}$.
\end{defn}

\begin{defn}
Let $\varphi\colon F_{Y}\to F_{X}$ be a non-degenerate homomorphism.
We define a functor $\mathcal{F}_{\varphi}$ from $Y$-labeled graphs
to $X$-labeled graphs by sending $y$-labeled edges to $\varphi\left(y\right)$-labeled
paths. Formally, let $\Delta$ be a $Y$-labeled graph and let $E_{0}=\left\{ e\in E\left(\Delta\right)|l\left(e\right)\in Y\right\} $
be an orientation of $\Delta$, namely, $E\left(\Delta\right)=E_{0}\sqcup\{\overline{e}\,|\,e\in E_{0}\}$.
For every $e\in E\left(\Delta\right)$ let $n_{e}\in\mathbb{N}$ be
the length of the word $\varphi\left(l\left(e\right)\right)\in F_{X}$
plus one. We consider $V\left(\Delta\right)$ as a graph without edges,
take the disjoint union of graphs $\bigsqcup_{e\in E_{0}}\Gamma_{\varphi\left(l\left(e\right)\right)}\sqcup V\left(\Delta\right)$
and for every $e\in E_{0}$ glue $1\in V\left(\Gamma_{\varphi\left(l\left(e\right)\right)}\right)$\textcolor{red}{{}
}to $\iota\left(e\right)\in V\left(\Delta\right)$, and $n_{e}\in V\left(\Gamma_{\varphi\left(l\left(e\right)\right)}\right)$
to $\iota\left(\overline{e}\right)\in V\left(\Delta\right)$. As for
functoriality, if $f\colon\Delta\to\Xi$ is a morphism of $Y$-labeled
graphs, $\mathcal{F}_{\varphi}f$ is defined as follows: every edge
in $\mathcal{F}_{\varphi}\left(\Delta\right)$ belongs to a path $\mathcal{F}_{\varphi}\left(e\right)$
for some $e\in E(\Delta)$, and we define $\left(\mathcal{F}_{\varphi}f\right)\left(\mathcal{F}_{\varphi}\left(e\right)\right)=\mathcal{F}_{\varphi}\left(f\left(e\right)\right)$.
\end{defn}

The motivation for $\mathcal{F}_{\varphi}$ is topological: thinking
of $\Delta,R_{Y},R_{X}$ as topological spaces and of $l\colon\Delta\to R_{Y},\varphi:R_{Y}\to R_{X}$
as continuous functions, we would like to think of $\Delta$ as an
$X$-labeled graph with the label function $\varphi\circ l$. The
problem is that $\varphi\circ l$ does not send edges to edges, and
we mend this by splitting edges in $\Delta$ to paths representing
their images in $R_{X}$.
\begin{rem}
\label{rem:Let--we}For $H\leq F_{Y}$ we notice that $\core(\mathcal{F}_{\varphi}\Gamma_{Y}(H))=\Gamma_{X}\left(\varphi\left(H\right)\right)$.
\end{rem}

\begin{defn}[Stencil]
Let $\Gamma$ be an $Y$-labeled graph, and $\varphi\colon F_{Y}\to F_{X}$
a non-degenerate homomorphism. We say that the pair $\left(\varphi,\Gamma\right)$
is a \textit{stencil} iff $\mathcal{F}_{\varphi}\left(\Gamma\right)$
is a folded graph. Notice that if $\Gamma$ is not folded, then $\mathcal{F}_{\varphi}\left(\Gamma\right)$
is not folded for any $\varphi$.
\end{defn}

\begin{defn}
Let $\tau\colon F_{X}\backslash\left\{ 1\right\} \to X\cup X^{-1}$
be the function returning the last letter of a reduced word. For reduced
words $u,v$ in a free group, we write $u\cdot v$ to indicate that
there is no cancellation in their concatenation, namely $\tau(u)\neq\tau(v^{-1})$.
\end{defn}

We leave the proof of Lemmas \ref{lem:Let--be} and \ref{lem:Let-and-}
as easy exercises for the reader.
\begin{lem}
\label{lem:Let--be}Let $\Gamma$ be an $Y$-labeled graph and $\varphi\colon F_{Y}\to F_{X}$
a non-degenerate homomorphism. Then $\left(\varphi,\Gamma\right)$
is a stencil iff $\Gamma$ is folded and for every $x.y\in W\left(\Gamma\right)$
we have $\varphi\left(x\right)\cdot\varphi\left(y\right)^{-1}$ (i.e.\ $\tau\left(\varphi\left(x\right)\right)\neq\tau\left(\varphi\left(y\right)\right)$).
\end{lem}

\begin{prop}
\label{prop:compos}Let $\varphi\colon F_{Z}\to F_{Y}$ and $\psi\colon F_{Y}\to F_{X}$
be non-degenerate homomorphisms. The equality $\mathcal{F}_{\psi\circ\varphi}=\mathcal{F}_{\psi}\circ\mathcal{F}_{\varphi}$
holds iff $(\psi,\Gamma_{\varphi(x)})$ is a stencil for every $x\in Z$.
\end{prop}

\begin{proof}
As $\mathcal{F}_{\varphi}$ is defined by replacing edges by paths
and gluing, it is enough to consider the graph $\Gamma_{x}$ with
a single edge labeled $x\in Z$. Both $\mathcal{F}_{\psi\circ\varphi}\left(\Gamma_{x}\right)$
and $\mathcal{F}_{\psi}\circ\mathcal{F}_{\varphi}\left(\Gamma_{x}\right)$
are paths whose labels equal $\psi\left(\varphi\left(x\right)\right)$,
but $\mathcal{F}_{\psi\circ\varphi}\left(\Gamma_{x}\right)$ is always
folded whereas $\mathcal{F}_{\psi}\circ\mathcal{F}_{\varphi}\left(\Gamma_{x}\right)$
may not be. In fact, $\mathcal{F}_{\varphi}\left(\Gamma_{x}\right)=\Gamma_{\varphi\left(x\right)}$
and so by definition $\mathcal{F}_{\psi}\circ\mathcal{F}_{\varphi}\left(\Gamma_{x}\right)$
is folded iff $(\psi,\Gamma_{\varphi\left(x\right)})$ is a stencil.
\end{proof}
\begin{lem}
\label{lem:Let-and-}Let $\varphi$ and $\psi$ be homomorphism as
in Proposition \ref{prop:compos} (such that\ $\mathcal{F}_{\psi\circ\varphi}=\mathcal{F}_{\psi}\circ\mathcal{F}_{\varphi}$)
and let $x\in Z$. Then $\tau\psi\left(\varphi\left(x\right)\right)=\tau\psi\left(\tau\varphi\left(x\right)\right)$.
\end{lem}

\begin{defn}[FGR]
The objects of the category \emph{Free Groups with Restrictions}
($\mathbf{FGR}$) are pairs $\left(Y,N\right)$ where $Y$ is a set
of ``generators'' and $N\subseteq W_{X}$ a set of ``restrictions''.
A morphism $\varphi\in\hom\left(\left(Y,N\right),\left(X,M\right)\right)$
is a group homomorphism $\varphi\colon F_{Y}\rightarrow F_{X}$ with
the following properties:
\begin{enumerate}
\item[(i)]  For every $x\in Y$, $\varphi\left(x\right)\neq1$ ($\varphi$ is
non-degenerate).
\item[(ii)] For every $x\in Y$, $W(\Gamma_{\varphi\left(x\right)})\subseteq M$.
\item[(iii)] For every $x.y\in N$, $\varphi(x)\cdot\varphi(y)^{-1}$ (i.e.\ $\tau\left(\varphi\left(x\right)\right)\neq\tau\left(\varphi\left(y\right)\right)$).
\item[(iv)] For every $x.y\in N$, $\tau\left(\varphi\left(x\right)\right).\tau\left(\varphi\left(y\right)\right)\in M$.\footnote{Technically, (iv) implies (iii), as $M\subseteq W_{Y}$ and $x.x\notin W_{Y}$.}
\end{enumerate}
\end{defn}

A main motivation for this definition is the second part of the next
Proposition.
\begin{prop}
\begin{enumerate}
\item $\mathbf{FGR}$ is a category, i.e.\ composition of morphisms is
a morphism.
\item Any two composable FGR morphisms $\varphi,\psi$ satisfy $\mathcal{F}_{\psi\circ\varphi}=\mathcal{F}_{\psi}\circ\mathcal{F}_{\varphi}$.
\end{enumerate}
\end{prop}

\begin{proof}
Let $(X_{1},N_{1})\overset{\varphi}{\longrightarrow}(X_{2},N_{2})\overset{\psi}{\longrightarrow}(X_{3},N_{3})$.
For any $x\in X_{1}$ we have $W(\Gamma_{\varphi\left(x\right)})\subseteq N_{2}$
since $\varphi$ satisfies (ii). Since $\psi$ satisfies (iii), $(\psi,\Gamma_{\varphi(x)})$
is a stencil, so that $\mathcal{F}_{\psi\circ\varphi}=\mathcal{F}_{\psi}\circ\mathcal{F}_{\varphi}$
by Proposition \ref{prop:compos}. This also implies $\len\psi(\varphi\left(x\right))\geq\len\varphi\left(x\right)\geq1$,
so that $\psi\!\circ\!\varphi$ satisfies (i). As $\Gamma_{\psi\circ\varphi(x)}=\mathcal{F}_{\psi}(\Gamma_{\varphi(x)})$,
we have
\[
W\left(\Gamma_{\psi\circ\varphi\left(x\right)}\right)=\left[\bigcup\nolimits _{\smash{z.y\in W(\Gamma_{\varphi(x)})}}W(\Gamma_{\psi(z)})\right]\cup\left\{ \tau\psi\left(z\right).\tau\psi\left(y\right)\mid z.y\in W(\Gamma_{\varphi\left(x\right)})\right\} 
\]
The part in $\left[\ \ \right]$ is contained in $N_{3}$ since $\psi$
satisfies (ii), and the rest is contained in $N_{3}$ since $\varphi$
satisfies (ii) and $\psi$ satisfies (iv), hence $\psi\!\circ\!\varphi$
satisfies (ii). Finally, whenever $x.y\in N_{1}$, it follows from
(iv) for $\varphi$ that $\tau\varphi\left(x\right).\tau\varphi\left(y\right)\in N_{2}$,
and from (iv) for $\psi$ that $\tau\psi(\tau\varphi(x)).\tau\psi(\tau\varphi(y))\in N_{3}$.
Using Lemma \ref{lem:Let-and-} we see that $\tau\psi(\varphi(x)).\tau\psi(\varphi(y))\in N_{3}$,
hence $\psi\!\circ\!\varphi$ satisfies (iv), and thus also (iii).
\end{proof}
\begin{lem}
An FGR morphism $\varphi\colon\left(Y,N\right)\to\left(X,M\right)$
is an FGR-isomorphism iff $\varphi|_{Y\cup Y^{-1}}\colon Y\cup Y^{-1}\to X\cup X^{-1}$
is a bijection that commutes with inversion and $M=\{\varphi(x).\varphi(y)\,|\,x.y\in N\}$.
\end{lem}

\begin{proof}
We show that if $\varphi\colon\left(Y,N\right)\to\left(X,M\right)$
is an FGR-isomorphism then $\varphi|_{Y\cup Y^{-1}}\colon Y\cup Y^{-1}\to X\cup X^{-1}$,
and the rest is immediate. Let $\varphi$ and $\psi$ be composable
FGR morphisms. As $\left(\psi,\Gamma_{\varphi\left(x\right)}\right)$
is a stencil, for every $x\in Y$ we have $\length\left(\psi\circ\varphi\left(x\right)\right)\ge\length\left(\varphi\left(x\right)\right)$.
Let $x\in Y$ and let $\varphi$ be an isomorphism and $\psi$ its
inverse. So $\length\left(\varphi\left(x\right)\right)\leq\length\left(\psi\circ\varphi\left(x\right)\right)=\length\left(x\right)=1$,
and on the other hand $\varphi\left(x\right)\neq1$ implies $\length\left(\varphi\left(x\right)\right)=1$,
namely $\varphi\left(x\right)\in X\cup X^{-1}$.
\end{proof}

\subsection{\label{subsec:Partition-of-Hom}Partition of Hom in FGR}

Let $X$ be a countably infinite set, $Y$ a finite set and $(Y,N_{Y})$
an FGR object. We present a recursive partition of $\hom\left((Y,N_{Y}),(X,W_{X})\right)$.
The ultimate goal of this partition is decomposing free group morphisms
$F_{Y}\to F_{X}$, and we notice that $\hom\left(\left(Y,\varnothing\right),(X,W_{X})\right)$
are all the non-degenerate homomorphisms in $\hom\left(F_{Y},F_{X}\right)$.
In order to perform this recursive process we need to consider morphisms
with a general restriction set $N_{Y}$ in the domain, but it is not
necessary to consider a general $N_{X}$ in place of $W_{X}$ (and
in fact, the method presented does not work for a general restriction
set $N_{X}$).

Let $\varphi\in\hom((Y,N_{Y}),(X,W_{X}))$. If $x.y\in N_{Y}$ then
$\varphi(x)\cdot\varphi(y^{-1})$ by definition. Let $x.y\in W_{Y}\backslash N_{Y}$.
There are five possible types of cancellation in the product $\varphi(x)\varphi(y^{-1})$.
Denote $u=\varphi\left(x\right)$, $v=\varphi\left(y\right)$, and
let $t$ be the maximal subword canceled in $uv^{-1}$, so that $u=u_{0}\cdot t$,
$v=v_{0}\cdot t$, and $uv^{-1}=u_{0}\cdot v_{0}^{-1}$. The five
types of cancellation are:
\begin{enumerate}
\item No cancellation, namely $t=1$.
\item $u$ and $v^{-1}$ do not absorb one another, namely $t,u_{0},v_{0}\neq1$
and $uv^{-1}=u_{0}\cdot v_{0}^{-1}$.
\item $u$ absorbs $v^{-1}$, namely $t\neq1$, $v_{0}=1,u_{0}\neq1$ and
$uv^{-1}=u_{0}$.
\item $v^{-1}$ absorbs $u$, namely $t\neq1$, $v_{0}\neq1,u_{0}=1$ and
$uv^{-1}=v_{0}^{-1}$.
\item $u$ and $v^{-1}$ cancel each other out, namely $u=v$.
\end{enumerate}
If $y=x^{-1}$ then only the first two types can occur, and if $x^{-1}.y^{-1}\in N_{Y}$
only the first four. For $\varphi\in\hom\left((Y,N_{Y}),(X,W_{X})\right)$
and $x.y\in W_{Y}\backslash N_{Y}$ such that $\varphi(x)\varphi(y^{-1})$
is of cancellation type $i$, we define an FGR object $(U_{i},N_{U_{i}})$
and a so-called \emph{folding morphism} $\psi_{x.y}^{i}\in\hom((Y,N_{Y}),(U_{i},N_{U_{i}}))$,
such that\ $\varphi$ factors through $\psi_{x.y}^{i}$. Namely,
there is an FGR morphism $\varphi'\in\hom\left((U_{i},N_{U_{i}}),(X,W_{X})\right)$
with $\varphi=\varphi'\circ\psi_{x.y}^{i}$. We describe $U_{i},N_{U_{i}},\psi_{x.y}^{i},\varphi'$
for each type:
\begin{enumerate}
\item $U_{1}=Y$, $\psi_{x.y}^{1}=\id$, $N_{U_{1}}=N_{Y}\cup\left\{ x.y\right\} $,
and $\varphi'=\varphi$.
\item There are two cases here, depending on whether $x=y^{-1}$ or not:

\uline{\mbox{$x\neq y^{-1}$}}: $U_{2}=Y\sqcup\left\{ s\right\} $
and $N_{U_{2}}=\left\{ \tau\psi_{x.y}^{2}(r).\tau\psi_{x.y}^{2}(z)|r.z\in N_{Y}\right\} \cup\left\{ x.s^{-1},y.s^{-1},x.y\right\} $
where 
\begin{align*}
\psi_{x.y}^{2}\left(z\right) & =\begin{cases}
xs & z=x\\
ys & z=y\\
z & z\in Y\backslash\{x^{\pm1},y^{\pm1}\}
\end{cases} & \varphi'\left(z\right) & =\begin{cases}
u_{0} & z=x\\
v_{0} & z=y\\
t & z=s\\
\varphi\left(z\right) & z\in U_{2}\backslash\{x^{\pm1},y^{\pm1},s^{\pm1}\}.
\end{cases}
\end{align*}
\uline{\mbox{$x=y^{-1}$}}: $U_{2}=Y\sqcup\left\{ s\right\} $
and $N_{U_{2}}\negthickspace=\negthickspace\{\tau\psi_{x.y}^{2}(r).\tau\psi_{x.y}^{2}(z)|r.z\in N_{Y}\}\cup\{x.s^{-1},y.s^{-1},x.y\}$
where
\begin{align*}
\psi_{x.y}^{2}\left(z\right) & =\begin{cases}
s^{-1}xs & z=x\\
z & z\in Y\backslash\{x^{\pm1}\}
\end{cases} & \varphi'\left(z\right) & =\begin{cases}
u_{0} & z=x\\
t & z=s\\
\varphi\left(z\right) & z\in U_{2}\backslash\{x^{\pm1},s^{\pm1}\}.
\end{cases}
\end{align*}

\item $U_{3}=Y$ and $N_{U_{3}}=\left\{ \tau\psi_{x.y}^{3}\left(r\right).\tau\psi_{x.y}^{3}\left(z\right)|r.z\in N_{Y}\right\} \cup\left\{ x.y^{-1}\right\} $
where
\begin{align*}
\psi_{x.y}^{3}\left(z\right) & =\begin{cases}
xy & z=x\\
z & z\in Y\backslash\{x^{\pm1}\}
\end{cases} & \varphi'\left(z\right) & =\begin{cases}
u_{0} & z=x\\
\varphi\left(z\right) & z\in U_{3}\backslash\{x^{\pm1}\}.
\end{cases}
\end{align*}
\item $U_{4}=Y$ and $N_{U_{4}}=\left\{ \tau\psi_{x.y}^{4}\left(r\right).\tau\psi_{x.y}^{4}\left(z\right)|r.z\in N_{Y}\right\} \cup\left\{ x^{-1}.y\right\} $
where
\begin{align*}
\psi_{x.y}^{4}\left(z\right) & =\begin{cases}
yx & z=y\\
z & z\in Y\backslash\{y^{\pm1}\}
\end{cases} & \varphi'\left(z\right) & =\begin{cases}
v_{0} & z=y\\
\varphi\left(z\right) & z\in U_{4}\backslash\{y^{\pm1}\}.
\end{cases}
\end{align*}
\item $U_{5}=Y-\left\{ y\right\} $ and $N_{U_{5}}=\left\{ \tau\psi_{x.y}^{5}\left(r\right).\tau\psi_{x.y}^{5}\left(z\right)|r.z\in N_{Y}\right\} $
where
\begin{align*}
\psi_{x.y}^{5}\left(z\right) & =\begin{cases}
x & z=y\\
z & z\in Y\backslash\{y^{\pm1}\}
\end{cases} & \varphi'\left(z\right) & =\varphi\left(z\right)\;\forall z\in U_{5}.
\end{align*}
\end{enumerate}
With $x.y$ fixed, let $\psi_{x.y}^{i*}\colon\hom\left(\left(U_{i},N_{U_{i}}\right),\left(X,W_{X}\right)\right)\to\hom\left(\left(Y,N_{Y}\right),\left(X,W_{X}\right)\right)$
be the induced function $\varphi'\mapsto\varphi'\circ\psi_{x.y}^{i}$.
We obtain that $\hom\left(\left(Y,N_{Y}\right),\left(X,W_{X}\right)\right)$
can be presented as a disjoint union
\[
\hom\left(\left(Y,N_{Y}\right),\left(X,W_{X}\right)\right)=\bigsqcup\nolimits _{i=1}^{5}\psi_{x.y}^{i*}\left(\hom\left(\left(U_{i},N_{U_{i}}\right),\left(X,W_{X}\right)\right)\right).
\]
For every $(U_{i},N_{U_{i}})$, if $N_{U_{i}}\neq W_{U_{i}}$ we continue
recursively, choosing an edge $x.y\in W_{U_{i}}\backslash N_{U_{i}}$
and partitioning $\hom\left((U_{i},N_{U_{i}}),\left(X,W_{X}\right)\right)$
accordingly. We now show that no matter which edges are chosen, this
process terminates:
\begin{thm}
\label{thm:decomposition}Any $\varphi\in\hom((Y,N_{Y}),(X,W_{X}))$
decomposes as $\varphi=\varphi'\circ\psi_{k}\circ\dots\circ\psi_{1}$
such that
\begin{enumerate}
\item The morphisms $\psi_{j}\colon U_{j-1}\rightarrow U_{j}$ $(1\leq j\leq k$,
$U_{0}=Y)$ are folding morphisms.
\item $N_{U_{k}}=W_{U_{k}},$ and in particular $\varphi'(x)\cdot\varphi'(y)^{-1}$
for $x\neq y\in U_{k}\cup U_{k}^{-1}$.
\end{enumerate}
\end{thm}

\begin{proof}
Let \emph{$\mathcal{O}$ }be the collection of all FGR morphisms from
some $(Y,N_{Y})$ with $Y$ finite into $(X,W_{X})$. Define a height
function $h\colon\mathcal{O}\to\mathbb{N}\times\mathbb{N}$ by $h\left(\varphi\right)=\big(\sum_{y\in Y}\length\left(\varphi\left(y\right)\right),\left|W_{Y}\backslash N_{Y}\right|\big)$,
and consider $\mathbb{N}\times\mathbb{N}$ with the lexicographic
order. It is straightforward to verify that $h\left(\varphi'\right)<h\left(\varphi\right)$
for every folding morphism $\psi_{x.y}^{i}$ and $\varphi=\varphi'\circ\psi_{x.y}^{i}$.
Thus, the decomposition process ends in a finite amount of steps.
\end{proof}
\begin{prop}[The triangle rule]
\label{prop:triangle}Let $(Y,N_{Y})$ be an FGR object, and $x.y\in N_{Y}$
and $z\in Y\cup Y^{-1}$ be such that $z.y,z.x\in W_{Y}\backslash N_{Y}$.
Then 
\[
\hom\left(\left(Y,N_{Y}\right),\left(X,W_{X}\right)\right)=\psi_{z.y}^{1*}\left(\hom\left(\left(Y,N_{Y_{1}}\right),\left(X,W_{X}\right)\right)\right)\cup\psi_{z.x}^{1*}\left(\hom\left(\left(Y,N_{Y_{2}}\right),\left(X,W_{X}\right)\right)\right)
\]
(not a disjoint union), where $N_{Y_{1}}=N_{Y}\cup\left\{ z.y\right\} $
and $N_{Y_{2}}=N_{Y}\cup\left\{ z.x\right\} $.
\end{prop}

\begin{proof}
Let $\varphi\colon(Y,N_{Y})\to(X,W_{X})$. As $\tau\left(\varphi\left(x\right)\right)\!=\!\tau\left(\varphi\left(z\right)\right)\!=\!\tau\left(\varphi\left(y\right)\right)$
contradicts $x.y\!\in\!N_{Y}$,~either $\tau\left(\varphi\left(x\right)\right)\neq\tau\left(\varphi\left(z\right)\right)$
or $\tau\left(\varphi\left(y\right)\right)\neq\tau\left(\varphi\left(z\right)\right)$,
hence $\varphi$ decomposes as $\varphi\!=\!\varphi'\circ\psi_{z.y}^{1}$
or $\varphi\!=\!\varphi'\circ\psi_{z.x}^{1}$.
\end{proof}

\subsection{The Core functor}

The flowing easy claim leads to some very helpful observations.
\begin{claim}
\label{claim:pushout}Let $\Gamma\to\Delta$ be a surjective graph
morphism. Let $\Gamma'$ be a graph resulting from folding two edges
in $\Gamma$, and $\Delta'$ the pushout of $\Gamma'\leftarrow\Gamma\rightarrow\Delta$
(which is obtained by folding the images of these edges in $\Delta$,
if they are different from one another \cite{MR695906}). Then $\Gamma'\to\Delta'$
is surjective.
\end{claim}

Let $\Gamma$ be a finite $Y$-labeled graph. One obtains $\core\Gamma$
by a finite sequence of folding and trimming, and one can perform
foldings first and only then trimmings, as when a folded graph is
trimmed it remains folded. Following the claim, if $\Gamma\to\Delta$
is a surjective graph morphism and $\core\Gamma$ is obtained from
$\Gamma$ without trimming then $\core\Gamma\to\core\Delta$ is also
surjective.
\begin{cor}
\label{rem:If--is}If $\Gamma$ is a core graph and $\left(\varphi,\Gamma\right)$
a stencil, then $\mathcal{F}_{\varphi}\left(\Gamma\right)$ is a core
graph, namely, $\core\mathcal{F}_{\varphi}\left(\Gamma\right)=\mathcal{F}_{\varphi}\left(\Gamma\right)$.
If $\Gamma,\,\Delta$ are core graphs, $\left(\varphi,\Gamma\right)$
is a stencil and $\Gamma\to\Delta$ is onto, then $\core\mathcal{F}_{\varphi}\left(\Gamma\right)\to\core\mathcal{F}_{\varphi}\left(\Delta\right)$
is onto.
\end{cor}

We recall a Lemma from \cite{MR3211804}:
\begin{lem}[\cite{MR3211804}]
\label{lem:Ori} Let $\Gamma$ be a finite $X$-labeled graph such
that for every vertex $v$, except for possibly the base-point, there
are $e,e'\in\iota^{-1}(v)$ with $l(e)\neq l(e')$. Then $\core\Gamma$
is obtained from $\Gamma$ by foldings alone (i.e.\ without trimming).
\end{lem}

\begin{cor}
\label{cor:Let--be}Let $H\leq K\leq F_{Y}$ be subgroups such that\ there
is a non-trivial cyclically reduced word in $H$. If the morphism
$\Gamma_{Y}\left(H\right)\to\Gamma_{Y}\left(K\right)$ is onto then
$\Gamma_{Y}\left(uHu^{-1}\right)\to\Gamma_{Y}\left(uKu^{-1}\right)$
is onto for every $u\in F_{Y}$.
\end{cor}

\begin{proof}
Let $u\in F_{Y}$ be a reduced word. We construct the graph $\Gamma_{Y}^{u}\left(H\right)$
by taking the graph $\Gamma_{Y}\left(H\right)\sqcup\Gamma_{u}$, gluing
$\length(u)+1\in V\left(\Gamma_{u}\right)$ to the base point of $\Gamma_{Y}\left(H\right)$
and setting the new base point to be $1\in V\left(\Gamma_{u}\right)$.
We notice that $\pi_{1}\left(\Gamma_{Y}^{u}\left(H\right)\right)=uHu^{-1}$
and that $\Gamma_{Y}^{u}\left(H\right)\to\Gamma_{Y}^{u}\left(K\right)$
is onto. Since $H$ and $K$ contain a cyclically reduce word the
degrees of the base points of $\Gamma_{Y}\left(H\right)$ and $\Gamma_{Y}\left(K\right)$
are at least two. From this we conclude that both $\Gamma_{Y}^{u}\left(H\right)$
and $\Gamma_{Y}^{u}\left(K\right)$ satisfy the conditions of Lemma
\ref{lem:Ori}. Finally an inductive application of Claim \ref{claim:pushout}
gives the result.
\end{proof}

\section{\label{sec:General-setting}Basis-independent surjectivity}

Returning to Conjectures \ref{conj:old}-\ref{conj:Xge3}, recall
that for $H\leq K\leq F_{Y}$, we seek to show that for every free
extension $F_{X}$ of $F_{Y}$ and every $\varphi\in\aut\left(F_{X}\right)$
the graph morphism $\Gamma_{X}\left(\varphi\left(H\right)\right)\to\Gamma_{X}\left(\varphi\left(K\right)\right)$
is surjective. Without loss of generality we may take $X$, from now
on, to be a fixed countably infinite set, and assume $Y\subseteq X$.
In Theorem \ref{thm:main} we even consider the set of non-degenerate
homomorphisms $\hom\left(\left(Y,\varnothing\right),(X,W_{X})\right)$:
this includes all injective homomorphisms $F_{Y}\hookrightarrow F_{X}$,
which in turn include all the restrictions of automorphisms of $F_{X}$.
Using Remark \ref{rem:Let--we} we translate the problem to showing
that all the graph morphisms $\{\core\mathcal{F}_{\varphi}\left(\Gamma_{Y}\left(H\right)\to\Gamma_{Y}\left(K\right)\right)\,|\,\varphi\in\hom\left(\left(Y,\varnothing\right),(X,W_{X})\right)\}$
are surjective. This is a special case of the following problem.
\begin{problem}
\label{prob:Given-a-morphism}Given a morphism $\Gamma\rightarrow\Delta$
between two\textcolor{red}{{} }$U$-labeled core graphs, and an FGR
object $\left(U,N_{U}\right)$, we call the pair $\left(\Gamma\to\Delta,\left(U,N_{U}\right)\right)$
a \textit{surjectivity problem}\textit{\emph{. We say that the problem
}}\textit{resolves positively }\textit{\emph{if }}all the morphisms
in the set $\mathscr{P}=\{\core\mathcal{F}_{\varphi}\left(\Gamma\to\Delta\right)\,|\,\varphi\in\hom\left(\left(U,N_{U}\right),(X,W_{X})\right)\}$
are surjective. Note that taking $(U,N_{U})=(Y,\varnothing)$ we recover
the problem of surjectivity w.r.t.\ every non-degenerate $\varphi\colon F_{U}\rightarrow F_{X}$,
and taking $(U,N_{U})=(Y,W(\Gamma))$ we obtain a trivial problem.
\end{problem}

Let us give a simple criterion for showing that one instance of Problem
\ref{prob:Given-a-morphism} is contained in another, or that the
problems are equivalent. For $U'$-labeled core graphs $\Gamma',\Delta'$
and $\psi\colon\left(U,N_{U}\right)\to\left(U',N_{U'}\right)$ an
FGR morphism such that\ $\mathcal{F}_{\psi}\left(\Gamma\rightarrow\Delta\right)=\Gamma'\rightarrow\Delta'$,
\begin{align*}
 & \negthickspace\negthickspace\negthickspace\negthickspace\negthickspace\negthickspace\negthickspace\negthickspace\negthickspace\negthickspace\negthickspace\negthickspace\negthickspace\negthickspace\left\{ \core\mathcal{F}_{\varphi}\left(\Gamma'\to\Delta'\right)\,\middle|\,\varphi\in\hom\left(\left(U',N_{U'}\right),\left(X,W_{X}\right)\right)\right\} \\
 & \ \ \ \ \ \ \ \ =\left\{ \core\mathcal{F}_{\varphi}\circ\mathcal{F}_{\psi}\left(\Gamma\to\Delta\right)\,\middle|\,\varphi\in\hom\left(\left(U',N_{U'}\right),\left(X,N_{X}\right)\right)\right\} \\
 & \ \ \ \ \ \ \ \ =\left\{ \core\mathcal{F}_{\varphi\circ\psi}\left(\Gamma\to\Delta\right)\,\middle|\,\varphi\in\hom\left(\left(U',N_{U'}\right),\left(X,N_{X}\right)\right)\right\} \\
 & \ \ \ \ \ \ \ \ \subseteq\left\{ \core\mathcal{F}_{\varphi}\left(\Gamma\to\Delta\right)\,\middle|\,\varphi\in\hom\left(\left(U,N_{U}\right),\left(X,N_{X}\right)\right)\right\} ,
\end{align*}
so that the problem $\left(\Gamma'\to\Delta',\left(U',N_{U'}\right)\right)$
is contained in $\left(\Gamma\to\Delta,\left(U,N_{U}\right)\right)$.
If $\psi$ is an FGR-isomorphism, then the two problems are equivalent.
\begin{rem}
\label{rem:Following-corollary-}Following corollary \ref{cor:Let--be}
it is enough to consider surjectivity problems where $\pi_{1}\left(\Gamma\right)$
contains a cyclically reduced word: otherwise, one can examine the
graph of a conjugate of $\pi_{1}\left(\Gamma\right)$ that contains
a cyclically reduced word.
\end{rem}

\begin{defn}
Let $\Gamma$ be a $Y$-labeled folded graph. An FGR object $(Y,N_{Y})$
is said to be a \textit{stencil space}\emph{ of }$\Gamma$ if $W\left(\Gamma\right)\subseteq N_{Y}$.
The reason for the name is that for any object $\left(X,N_{X}\right)$
and morphism $\varphi\in\hom\left((Y,N_{Y}),\left(X,N_{X}\right)\right)$,
the pair $\left(\varphi,\Gamma\right)$ is a stencil.
\end{defn}

We show a method for determining whether all the morphisms in a problem
$\left(\Gamma\to\Delta,\left(U,N_{U}\right)\right)$ are surjective.
We distinguish three cases:
\begin{enumerate}
\item $\Gamma\to\Delta$ is not surjective: clearly $\mathscr{P}$ resolves
negatively.
\item $\Gamma\to\Delta$ is surjective and $W\left(\Gamma\right)\subseteq N_{U}$,
i.e.\ $\left(U,N_{U}\right)$ is a stencil space of $\Gamma$: following
Corollary \ref{rem:If--is}, $\mathscr{P}$ resolves positively.
\item $\Gamma\to\Delta$ is surjective and $W\left(\Gamma\right)\backslash N_{U}\neq\varnothing$:
in this case we cannot resolve $\mathscr{P}$ immediately. We call
this an ambiguous case.
\end{enumerate}
When $\left(\Gamma\to\Delta,\left(U,N_{U}\right)\right)$ is an ambiguous
case, we cover it by five new problems, according to the partition
of $\hom$ presented in §\ref{subsec:Partition-of-Hom}: For any $x.y\in W\left(\Gamma\right)\backslash N_{U}$
we have
\begin{align*}
 & \left\{ \core\mathcal{F}_{\varphi}\left(\Gamma\to\Delta\right)\,\middle|\,\varphi\in\hom\left(\left(U,N_{U}\right),\left(X,N_{X}\right)\right)\right\} \\
 & \ \ \ \ \ \ \ \ =\bigcup{\!\vphantom{\big|}}_{i=1}^{5}\left\{ \core\smash{\mathcal{F}_{\varphi\circ\psi_{x.y}^{i}}}\left(\Gamma\to\Delta\right)\,\middle|\,\varphi\in\hom\left(\left(U_{i},N_{U_{i}}\right),\left(X,W_{X}\right)\right)\right\} \\
 & \ \ \ \ \ \ \ \ =\bigcup{\!\vphantom{\big|}}_{i=1}^{5}\left\{ \core\smash{\mathcal{F}_{\varphi}\circ\mathcal{F}_{\psi_{z.y}^{i}}}\left(\Gamma\to\Delta\right)\,\middle|\,\varphi\in\hom\left(\left(U_{i},N_{U_{i}}\right),\left(X,W_{X}\right)\right)\right\} \\
 & \ \ \ \ \ \ \ \ =\bigcup{\!\vphantom{\big|}}_{i=1}^{5}\left\{ \core\smash{\mathcal{F}_{\varphi}}\left(\core\smash{\mathcal{F}_{\psi_{z.y}^{i}}}\left(\Gamma\to\Delta\right)\right)\,\middle|\,\varphi\in\hom\left(\left(U_{i},N_{U_{i}}\right),\left(X,W_{X}\right)\right)\right\} ,
\end{align*}
where the last equality follows from Remark \ref{rem:Let--we}. Thus
we obtain five different subproblems $\big(\core\mathcal{F}_{\psi_{z.y}^{i}}\left(\Gamma\to\Delta\right),(U_{i},N_{U_{i}})\big)$,
and $\mathscr{P}$ resolves positively iff all the five subproblems
do so. If one of these five problems is of case 1 then $\left(\Gamma\to\Delta,\left(U,N_{U}\right)\right)$
resolves negatively, if all are of case 2 then $\left(\Gamma\to\Delta,\left(U,N_{U}\right)\right)$
resolves positively, and otherwise we continue recursively and split
each ambiguous case into five sub-subproblems. We now give a toy example
to show how this process may end.
\begin{example}
\label{exa:comutator}Let $\mathscr{P}$ be the surjectivity problem
with $U=\left\{ x,y\right\} $, $N_{U}=\{x.y^{-1},x^{-1}.y,\allowbreak y^{-1}.x^{-1},x.x^{-1}\}$,
$\Gamma=\Gamma_{U}\left(\left\langle xyx^{-1}y^{-1}\right\rangle \right)$
(the commutator) and $\Delta=\Gamma_{U}\left(\left\langle x,y\right\rangle \right)$.
The extension $\left\langle xyx^{-1}y^{-1}\right\rangle <\left\langle x,y\right\rangle $
is algebraic so we know that this surjectivity problem resolves positively.
We will show how one can conclude this using the method presented
above. We notice that $\Gamma\twoheadrightarrow\Delta$ and $x.y\in W\left(\Gamma\right)\backslash N_{U}$,
so $\mathscr{P}$ is an ambiguous case. We split into five subproblems
(some of the calculations are explicit and some omitted and left to
the reader):
\begin{enumerate}
\item $\mathscr{P}_{1}:$ $\psi_{x.y}^{1}=\id$, $U_{1}=U,\,N_{U_{1}}=\left\{ x.y^{-1},x^{-1}.y,y^{-1}.x^{-1},x.x^{-1},x.y\right\} $,
$\Gamma_{1}=\Gamma$,~$\Delta_{1}=\Delta$. The problem $\mathscr{P}_{1}$
resolves positively as $\left(U_{1},N_{U_{1}}\right)$ is a stencil
space of $\Gamma_{1}$ and $\Gamma_{1}\twoheadrightarrow\Delta_{1}$.
\item $\mathscr{P}_{2}$: $\psi_{x.y}^{2}\negmedspace:\negmedspace\begin{smallmatrix}x\mapsto xt\\
y\mapsto yt
\end{smallmatrix}$, $U_{2}=\left\{ x,y,t\right\} ,$ $N_{U_{2}}=\left\{ t.y^{-1},x^{-1}.t,y^{-1}.x^{-1},x.t^{-1},y.t^{-1},x.y\right\} $
,$\Gamma_{2}=\core\mathcal{F}_{\psi_{x.y}^{2}}\left(\Gamma\right)=\Gamma_{U_{2}}\left(\psi_{x.y}^{2}\left\langle xyx^{-1}y^{-1}\right\rangle \right)=\Gamma_{U_{2}}\left(\left\langle xtyx^{-1}t^{-1}y^{-1}\right\rangle \right),\Delta_{2}=\Gamma_{U_{2}}\left(\left\langle xt,yt\right\rangle \right)$.
Again $\mathscr{P}_{2}$ resolves positively as $\left(U_{2},N_{U_{2}}\right)$
is a stencil space of $\Gamma_{2}$ and $\Gamma_{2}\twoheadrightarrow\Delta_{2}.$
\item $\mathscr{P}_{3}$: $\psi_{x.y}^{3}\!:\!\begin{smallmatrix}x\mapsto x\phantom{y}\\
y\mapsto yx
\end{smallmatrix}$, $U_{3}=U$, $N_{U_{3}}=\left\{ x.y^{-1},x^{-1}.x,y^{-1}.x^{-1},y.x^{-1}\right\} $,
$\Gamma_{3}=\Gamma\twoheadrightarrow\Delta=\Delta_{3}$, and $x.y\in W\left(\Gamma_{3}\right)\backslash N_{U_{3}}$,
so this is an ambiguous case.
\item $\mathscr{P}_{4}$: $\psi_{x.y}^{4}\!:\!\begin{smallmatrix}x\mapsto xy\\
y\mapsto y\phantom{x}
\end{smallmatrix},$ $U_{4}=U$, $N_{U_{4}}=\left\{ y.y^{-1},x^{-1}.y,y^{-1}.x^{-1},x.y^{-1}\right\} ,$
$\Gamma_{4}=\Gamma\twoheadrightarrow\Delta=\,\Delta_{4}$, and $x.y\in W\left(\Gamma_{4}\right)\backslash N_{U_{4}}$
so this is an ambiguous case.
\item This case does not occur because $x^{-1}.y^{-1}\in N_{U}$.
\end{enumerate}
We notice first that the FGR-isomorphism $x\leftrightarrow y$ gives
an equivalence between $\mathscr{P}_{4}$ and $\mathscr{P}_{3}$,
and second, that $\mathscr{P}_{3}$ is identical to the original $\mathscr{P}$!
Now, it may seem that we are trapped in an infinite self-referential
loop: $\mathscr{P}$ resolves positively if $\mathscr{P}$ resolves
positively. Thankfully, Theorem \ref{thm:decomposition} rescues us
from this nightmare. Let $\varphi\in\hom\left(\left(U,N_{U}\right),\left(X,W_{X}\right)\right)$.
We decompose $\varphi$ by $x.y$ recursively, mirroring the splitting
of the four subproblems. Because the decomposition of the morphism
via FGR is finite, it can loop through problems $3$ and 4 only a
finite number of times until it arrives at a stencil space equivalent
to $\mathscr{P}_{1}$ or $\mathscr{P}_{2}$. Since $\mathscr{P}_{1}$
and $\mathscr{P}_{2}$ resolve positively, we conclude that so does
$\mathscr{P}$.
\end{example}

Not all surjectivity problems end with equivalent ambiguous problems.
In some cases this process produces more and more different ambiguous
problems and therefore the algorithm does not end with a conclusive
answer. In the case of the counterexample (§\ref{sec:The-counterexample}),
after a fair amount of splitting we end up with ambiguous cases equivalent
to or contained in problems we have previously encountered. In Section
§\ref{sec:TBD-finiteness} we discuss a property of the graph $\Gamma$
in a problem $\mathscr{P}$ that is a necessary condition for this
process to end.

\subsection{Change of coordinates}

With $X$ and $(Y,N_{Y})$ as before, in $\mathscr{P}=\{\core\mathcal{F}_{\varphi}\left(\Gamma\to\Delta\right)\,|\,\varphi\in\hom\left(\left(Y,N_{Y}\right),(X,W_{X})\right)\}$
there can be many repetitions, as different homomorphisms $\varphi$
may give rise to the same graph morphism $\core\mathcal{F}_{\varphi}\left(\Gamma\to\Delta\right)$.
We can exploit this in a way which shortens calculations significantly
in §\ref{sec:The-counterexample}.
\begin{defn}
A \emph{change of coordinates} consists of\\
\begin{minipage}[t]{0.7\columnwidth}%
\begin{enumerate}
\item FGR objects $(V,N_{V}),(U_{1},N_{1}),\dots,(U_{n},N_{n})$,
\item a group homomorphism $\sigma\colon F_{V}\rightarrow F_{Y}$ (not necessarily
an FGR morphism),
\item an FGR morphism $\psi_{i}\colon(V,N_{V})\to(U_{i},N_{i})$ for every
$i$,
\item an FGR morphism $\sigma_{i}\colon(Y,N_{Y})\to(U_{i},N_{i})$ for every
$i$,\medskip{}
\end{enumerate}
\end{minipage}\hspace*{\fill}%
\begin{minipage}[t]{0.2\columnwidth}%
\vspace{0.3em}
$\xymatrix{ & Y\ar[rd]^{\varphi}\ar[dd]^{\sigma_{i}}\\
V\ar[dr]_{\psi_{i}}\ar[ur]^{\sigma} &  & X\\
 & U_{i}\ar[ur]_{\varphi'}
}
$\vspace{-4em}
\end{minipage}\hspace*{\fill}\\
such that
\begin{enumerate}
\item $\pi_{1}\left(\Delta\right)\leq\im\sigma$,
\item $\sigma^{*}:\varphi\mapsto\varphi\circ\sigma$ takes $\hom\left(\left(Y,N_{Y}\right),\left(X,W_{X}\right)\right)$
to $\hom\left(\left(V,N_{V}\right),\left(X,W_{X}\right)\right)$,
\item $\sigma_{i}\circ\sigma=\psi_{i}$ for every $i$,
\item for every $\varphi\in\hom\left((Y,N_{Y}),(X,W_{X})\right)$ there
is an index $i$ and a morphism $\varphi'\in\hom\left(\left(U_{i},N_{i}\right),(X,W_{X})\right)$
such that $\varphi\circ\sigma=\varphi'\circ\psi_{i}$.
\end{enumerate}
\end{defn}

\begin{prop}
\label{prop:COC} The problem $\left(\Gamma\to\Delta,\left(Y,N_{Y}\right)\right)$
resolves positively if and only if $\left(\core\mathcal{F}_{\sigma_{i}}\left(\Gamma\to\Delta\right),\left(U_{i},N_{i}\right)\right)$
resolves positively for every $i$.
\end{prop}

Note that not all FGR morphisms in $\hom\left(\left(Y,N_{Y}\right),\left(X,W_{X}\right)\right)$
necessarily factor through one of the $\sigma_{i}$, meaning there
could be $\varphi\in\hom\left(\left(Y,N_{Y}\right),\left(X,W_{X}\right)\right)$
(as is the case in §\ref{sec:The-counterexample}) with no decomposition
as $\varphi=\varphi'\circ\sigma_{i}$ -- this is what makes the new
set of problems potentially simpler.
\begin{proof}
Let $\pi_{1}\left(\Gamma\right)=H$ and $\pi_{1}\left(\Delta\right)=K$.
The morphisms $\psi_{i}$ induce functions $\psi_{i}^{*}\colon\hom\left(\left(U_{i},N_{i}\right),\left(X,W_{X}\right)\right)\to\hom\left(\left(V,N_{V}\right),\left(X,W_{X}\right)\right)$,
and conditions $2,3,4$ imply $5:\im\sigma^{*}=\bigcup_{i}\im\psi_{i}^{*}$.
As $\Gamma\to\Delta$ means that $H\leq K$, for every homomorphism
$\varphi:F_{Y}\to F_{X}$ we have $\varphi\left(H\right)\leq\varphi\left(K\right)$,
hence there is a unique graph morphism $\Gamma_{X}\left(\varphi\left(H\right)\right)\to\Gamma_{X}\left(\varphi\left(K\right)\right)$.
We notice that: 
\begin{align*}
\mathscr{P}= & \left\{ \core\mathcal{F}_{\varphi}\left(\Gamma_{Y}\left(H\right)\to\Gamma_{Y}\left(K\right)\right)|\varphi\in\hom\left(\left(Y,N_{Y}\right),\left(X,W_{X}\right)\right)\right\} \\
= & \left\{ \Gamma_{X}\left(\varphi\left(H\right)\right)\to\Gamma_{X}\left(\varphi\left(K\right)\right)|\varphi\in\hom\left(\left(Y,N_{Y}\right),\left(X,W_{X}\right)\right)\right\} \\
\overset{{\scriptscriptstyle 1}}{=} & \left\{ \Gamma_{X}\left(\varphi\circ\sigma\left(\sigma^{-1}\left(H\right)\right)\right)\to\Gamma_{X}\left(\varphi\circ\sigma\left(\sigma^{-1}\left(K\right)\right)\right)|\varphi\in\hom\left(\left(Y,N_{Y}\right)\left(X,W_{X}\right)\right)\right\} \\
\overset{{\scriptscriptstyle 5}}{=} & \bigcup\nolimits _{i}\left\{ \Gamma_{X}(\varphi'\circ\psi_{i}(\sigma^{-1}(H)))\to\Gamma_{X}(\varphi'\circ\psi_{i}(\sigma^{-1}(K)))|\varphi'\in\hom\left(\left(U_{i},N_{i}\right),\left(X,W_{X}\right)\right)\right\} \\
= & \bigcup\nolimits _{i}\left\{ \core\mathcal{F}_{\varphi'}\left(\Gamma_{U_{i}}(\psi_{i}(\sigma^{-1}(H)))\to\Gamma_{U_{i}}(\psi_{i}(\sigma^{-1}(K)))\right)|\varphi'\in\hom\left(\left(U_{i},N_{i}\right),\left(X,W_{X}\right)\right)\right\} \\
\overset{{\scriptscriptstyle 3}}{=} & \bigcup\nolimits _{i}\left\{ \core\mathcal{F}_{\varphi'}\left(\Gamma_{U_{i}}(\sigma_{i}\circ\sigma(\sigma^{-1}(H)))\to\Gamma_{U_{i}}(\sigma_{i}\circ\sigma(\sigma^{-1}(K)))\right)|\varphi'\in\hom\left(\left(U_{i},N_{i}\right),\left(X,W_{X}\right)\right)\right\} \\
\overset{{\scriptscriptstyle 1}}{=} & \bigcup\nolimits _{i}\left\{ \core\mathcal{F}_{\varphi'}\left(\Gamma_{U_{i}}(\sigma_{i}(H))\to\Gamma_{U_{i}}(\sigma_{i}(K))\right)|\varphi'\in\hom\left(\left(U_{i},N_{i}\right),\left(X,W_{X}\right)\right)\right\} \\
= & \bigcup\nolimits _{i}\left\{ \core\mathcal{F}_{\varphi'}\left(\core\mathcal{F}_{\sigma_{i}}\left(\Gamma\to\Delta\right)\right)|\varphi'\in\hom\left(\left(U_{i},N_{i}\right),\left(X,W_{X}\right)\right)\right\} .\qedhere
\end{align*}
\end{proof}

\section{\label{sec:The-counterexample}The counterexample}

In this section we prove Theorem \ref{thm:main}. showing that $H=\langle bbaba^{-1}\rangle$
and $K=\langle b,aba^{-1}\rangle$ constitute a counterexample to
Conjectures \ref{conj:old}, \ref{conj:main}, \ref{conj:Xge3}.
\begin{thm*}[\ref{thm:main}]
The extension $H\leq K$ is not algebraic, but for every morphism
$\varphi\colon F_{\{a,b\}}\rightarrow F_{X}$ with $X$ arbitrary
and $\varphi(a),\varphi(b)\neq1$, the graph morphism $\Gamma_{X}\left(\varphi\left(H\right)\right)\rightarrow\Gamma_{X}\left(\varphi(K)\right)$
is surjective.
\end{thm*}
First, $H$ is a proper free factor of $K$, so in particular $H\leq K$
is not an algebraic extension. Following Section §\ref{sec:General-setting},
we need to show that the problem $\left(\Gamma_{\left\{ a,b\right\} }\left(H\right)\to\Gamma_{\left\{ a,b\right\} }\left(K\right),\left(\left\{ a,b\right\} ,\varnothing\right)\right)$
resolves positively. By a change of coordinate (Proposition \ref{prop:COC})
we replace this problem with eight problems that are easier to analyze.
Let $\left(V,N_{V}\right)=\left(\left\{ \alpha,\beta\right\} ,\varnothing\right)$,
and
\[
\sigma:F_{\left\{ \alpha,\beta\right\} }\to F_{\left\{ a,b\right\} },\quad\sigma\left(\alpha\right)=b,\quad\sigma\left(\beta\right)=aba^{-1}.
\]
We notice that $H<K\leq\im\sigma$. For any non-degenerate $\varphi\colon F_{\{a,b\}}\rightarrow F_{X}$,
the words $\varphi\left(b\right)=\varphi\circ\sigma\left(\alpha\right)$
and $\varphi\left(aba^{-1}\right)=\varphi\circ\sigma\left(\beta\right)$
are conjugate, hence there exist reduced words $\overline{x},\overline{y},\overline{u},\overline{v}\in F_{X}$
such that $\varphi(b)=\overline{x}\cdot\overline{u}\cdot\overline{v}\cdot\overline{x}^{-1}$,
$\varphi(aba^{-1})=\overline{y}\cdot\overline{v}\cdot\overline{u}\cdot\overline{y}^{-1}$
(in particular, $\overline{v}\overline{u}$ and $\overline{u}\overline{v}$
are cyclically reduced). By non-degeneracy we can also assume $\overline{u}\neq1$,
and if $\overline{v}=1$ then $\overline{u}$ is cyclically reduced.
We perform a change of coordinates according to the eight possible
cases, with $(U_{i},N_{i})$, $\psi_{i}$ and $\sigma_{i}$ being:\medskip{}
\\
\hspace*{\fill}%
\begin{tabular}{|c|c|c|c|c|c|c|c|}
\hline 
\# & $\overline{x}$ & $\overline{y}$ & $\overline{v}$ & $U_{i}$ & $N_{i}$ & $\psi_{i}(\alpha),\psi_{i}(\beta)$ & $\sigma_{i}(a),\sigma_{i}(b)$\tabularnewline
\hline 
\hline 
1 & $\negmedspace=\negmedspace1\negmedspace$ & $\negmedspace=\negmedspace1\negmedspace$ & $\negmedspace=\negmedspace1\negmedspace$ & $\{u\}$ & $\{u.u^{-1}\}$ & $u,u$ & $u,u$\tabularnewline
\hline 
2 & $\negmedspace\neq\negmedspace1\negmedspace$ & $\negmedspace=\negmedspace1\negmedspace$ & $\negmedspace=\negmedspace1\negmedspace$ & $\{y,u\}$ & $\{y.u^{-1},u.y,u.u^{-1}\}$ & $u,yuy^{-1}$ & $y,u$\tabularnewline
\hline 
3 & $\negmedspace=\negmedspace1\negmedspace$ & $\negmedspace\neq\negmedspace1\negmedspace$ & $\negmedspace=\negmedspace1\negmedspace$ & $\{x,u\}$ & $\{x.u^{-1},u.x,u.u^{-1}\}$ & $xux^{-1},u$ & $x^{-1},xux^{-1}$\tabularnewline
\hline 
4 & $\negmedspace\neq\negmedspace1\negmedspace$ & $\negmedspace\neq\negmedspace1\negmedspace$ & $\negmedspace=\negmedspace1\negmedspace$ & $\{x,u,y\}$ & ${\displaystyle \left\{ {x.u^{-1},u.x,u.u^{-1},\atop y.u^{-1},u.y,y.x}\right\} }$ & $xux^{-1},yuy^{-1}$ & $yx^{-1},xux^{-1}$\tabularnewline
\hline 
5 & $\negmedspace=\negmedspace1\negmedspace$ & $\negmedspace=\negmedspace1\negmedspace$ & $\negmedspace\neq\negmedspace1\negmedspace$ & $\{v,u\}$ & $\{v.u^{-1},u.v^{-1}\}$ & $uv,vu$ & $u^{-1},uv$\tabularnewline
\hline 
6 & $\negmedspace\neq\negmedspace1\negmedspace$ & $\negmedspace=\negmedspace1\negmedspace$ & $\negmedspace\neq\negmedspace1\negmedspace$ & $\{v,u,y\}$ & $\negmedspace\{v.u^{-1},u.v^{-1},y.v^{-1},u.y\}\negmedspace$ & $uv,yvuy^{-1}$ & $yv,uv$\tabularnewline
\hline 
7 & $\negmedspace=\negmedspace1\negmedspace$ & $\negmedspace\neq\negmedspace1\negmedspace$ & $\negmedspace\neq\negmedspace1\negmedspace$ & $\{v,u,x\}$ & $\negmedspace\{v.u^{-1},u.v^{-1},x.u^{-1},v.x\}\negmedspace$ & $xuvx^{-1},vu$ & $u^{-1}x^{-1},xuvx^{-1}$\tabularnewline
\hline 
8 & $\negmedspace\neq\negmedspace1\negmedspace$ & $\negmedspace\neq\negmedspace1\negmedspace$ & $\negmedspace\neq\negmedspace1\negmedspace$ & $\negmedspace\{v,u,y,x\}\negmedspace$ & $\left\{ {\displaystyle {v.u^{-1},u.v^{-1},x.u^{-1},\atop v.x,\,y.v^{-1},u.y}}\right\} $ & $\negmedspace xuvx^{-1},yvuy^{-1}\negmedspace$ & $\negmedspace yu^{-1}x^{-1},xuvx^{-1}\negmedspace\negmedspace$\tabularnewline
\hline 
\end{tabular}\hspace*{\fill}\medskip{}

For every $1\leq i\leq8$ we denote $\Gamma_{i}=\Gamma\left(\sigma_{i}\left(H\right)\right)$
and $\Delta_{i}=\Gamma\left(\sigma_{i}\left(K\right)\right)$. We
obtain eight problems, and we proceed to split them and identify all
the stencil spaces. We index the cases in the following manner: if
Case $i$ is not a stencil space it splits into five cases Case $i.1,i.2,\dots i.5$.
For each case, the morphism $\psi_{i.j}$ is the folding morphism
this subset of morphisms factors through. The co-domain of $\psi_{i.j}$
is the FGR object $(U_{i.j},N_{i.j})$, and the graphs are indexed
by $\Gamma_{i.j}=\core\mathcal{F}_{\psi_{i.j}}\left(\Gamma_{i}\right)$
and $\Delta_{i.j}=\core\mathcal{F}_{\psi_{i.j}}\left(\Delta_{i}\right)$,
and the $\Delta$ graphs are depicted in Figure \ref{fig:The-graphs-}.
In each case either $W(\Gamma_{i})\backslash N_{i}$ is empty, hence
$(U_{i},N_{i})$ is a stencil space of $\Gamma_{i}$ and we can resolve
the subproblem, or it is not empty, and then we continue splitting.
Remark \ref{rem:Following-corollary-} allows us to conjugate, or
equivalently to change the base points of the graphs.
\begin{casenv}
\item [Case 5.]Here $\Gamma_{5}=\Gamma(\langle uvuvvu\rangle)\twoheadrightarrow\Delta_{5}$
and $W\left(\Gamma_{5}\right)\backslash N_{5}=\{u.u^{-1},v.v^{-1}\}$.
We split by $u.u^{-1}$.
\item [Case 5.1.]$\psi_{5.1}=\psi_{u.u^{-1}}^{1}$, $N_{5.1}=\{v.u^{-1},u.v^{-1},u.u^{-1}\}$,
$\Gamma_{5.1}=\Gamma_{5}$ and $\Delta_{5.1}=\Delta_{5}$ (hence $\Gamma_{5.1}\twoheadrightarrow\Delta_{5.1}$).
$W\left(\Gamma_{5.1}\right)\backslash N_{U_{5.1}}=\{v.v^{-1}\}$ (so
we must split by $v.v^{-1}$).
\item [Case 5.1.1.]$\psi_{5.1.1}=\psi_{v.v^{-1}}^{1}$, $N_{5.1.1}=\{v.u^{-1},u.v^{-1},u.u^{-1},v.v^{-1}\}$.
Again $\Gamma_{5}=\Gamma_{5.1.1},\Delta_{5.1.1}=\Delta_{5}$ and $W\left(\Gamma_{5.1}\right)\backslash N_{5.1}=\varnothing$,
hence the subcase resolves positively.
\item [Case 5.1.2.] $\psi_{5.1.2}=\psi_{v.v^{-1}}^{2}$, $v\mapsto t^{-1}vt$,
$N_{5.1.2}=\{t.u^{-1},u.t,u.u^{-1},v.t^{-1},t^{-1}.v^{-1},v.v^{-1}\}$;
$\Gamma_{5.1.2}=\Gamma\left(\left\langle ut^{-1}vtut^{-1}vvtu\right\rangle \right)\twoheadrightarrow\Delta_{5.1.2}$
and $W\left(\Gamma_{5.1.2}\right)\backslash N_{5.1.2}=\varnothing$
hence the subcase resolves positively.
\item [Case 5.2.] $\psi_{5.2}=\psi_{u.u^{-1}}^{2}$, $u\mapsto t^{-1}ut$,
$N_{5.2}=\{v.t,u.t^{-1},u.u^{-1},t^{-1}.u^{-1}\}$; $\Gamma_{5.2}=\Gamma\left(\left\langle t^{-t}utvt^{-1}utvvt^{-1}ut\right\rangle \right)$
is not cyclically reduced so we conjugate and set $\Gamma_{5.2}=\Gamma\left(\left\langle utvt^{-1}utvvt^{-1}u\right\rangle \right)\twoheadrightarrow\Delta_{5.2}$
and $W\left(\Gamma_{5.2}\right)\backslash N_{5.2}=\{v.v^{-1}\}$.
We split by $v.v^{-1}$.
\item [Case 5.2.1.] $\psi_{5.1}=\psi_{v.v^{-1}}^{1}$, $N_{U_{5.2.1}}=\{v.t,u.t^{-1},u.u^{-1},t^{-1}.u^{-1},v.v^{-1}\}$;
$\Gamma_{5.2.1}=\Gamma_{5.2},\Delta_{5.2.1}=\Gamma_{5.2}$ and $W\left(\Gamma_{5.2.1}\right)\backslash N_{5.2.1}=\varnothing$,
hence the subcase resolves positively.
\item [Case 5.2.2.] $\psi_{5.2.2}=\psi_{v.v^{-1}}^{2}$, $v\mapsto svs^{-1}$,
$N_{5.2.2}=\{s.t,u.t^{-1},u.u^{-1},t^{-1}.u^{-1},v.s^{-1},s^{-1}.v^{-1},v.v^{-1}\}$;
$\Gamma_{5.2.2}=\Gamma\left(\left\langle uts^{-1}vst^{-1}uts^{-1}vvst^{-1}u\right\rangle \right)\twoheadrightarrow\Delta_{5.2.2}$
and $W\left(\Gamma_{5.2.2}\right)\backslash N_{5.2.2}$, hence the
subcase resolves positively.
\item [Case 2.] Here $\Gamma_{2}=\Gamma\left(\left\langle uuyuy^{-1}\right\rangle \right)\twoheadrightarrow\Delta_{2}$
and $W\left(\Gamma_{2}\right)\backslash N_{2}=\{u^{-1}.y^{-1},u.y^{-1}\}$.
We notice that $u.u^{-1}\in N_{2}$, so we can use the triangle rule
(Proposition \ref{prop:triangle}) to deduce that every morphism factors
through either $\psi_{u^{-1}.y^{-1}}^{1}$ or $\psi_{u.y^{-1}}^{1}$.
\item [Case 2.1.]We have two cases: $\psi_{2.1}=\psi_{u^{-1}.y^{-1}}^{1}$
and $N_{2.1}=\{y.u^{-1},u.y,u.u^{-1},u^{-1}.y^{-1}\}$, or $\psi_{2.1'}=\psi_{u.y^{-1}}^{1}$
and $N_{2.1'}=\{y.u^{-1},u.y,u.u^{-1},u.y^{-1}\}$. These cases are
equivalent, as $\gamma\colon\left(U_{2},N_{2.1}\right)\to\left(U_{2},N_{2.1'}\right)$
defined by $u\mapsto u^{-1},y\mapsto y$ satisfies $\mathcal{F}_{\gamma}\left(\Gamma_{2.1}\to\Delta_{2.1}\right)=\Gamma_{2.1'}\to\Delta_{2.1'}$.
Finally, $\Gamma_{2.1}=\Gamma_{2},\,\Delta_{2.1}=\Delta_{2}$, and
$W\left(\Gamma_{2.1}\right)\backslash N_{2.1}=\{u.y^{-1}\}$.
\item [Case 2.1.1.] $\psi_{2.1.1}=\psi_{u.y^{-1}}^{1}$ and $N_{2.1.1}=\{y.u^{-1},u.y,u.u^{-1},u^{-1}.y^{-1},u.y^{-1}\}$;
$\Gamma_{2.1.1}=\Gamma_{2.1},\Delta_{2.1.1}=\Delta_{2.1}$ and $W\left(\Gamma_{2.1.1}\right)\backslash N_{2.1.1}=\varnothing$,
hence the subcase resolves positively.
\item [Case 2.1.2.] $\psi_{2.1.2}=\psi_{u.y^{-1}}^{2}$, $u\mapsto ut,\,y\mapsto t^{-1}y$
and $N_{2.1.2}=\{t.u^{-1},u^{-1}.y,t.y,t.u^{-1},u.t^{-1},\allowbreak y^{-1}.t^{-1},u.y^{-1}\}$;
$\Gamma_{2.1.2}=\Gamma\left(\left\langle utuyuty^{-1}t\right\rangle \right)\twoheadrightarrow\Delta_{2.1.2}$
and $W\left(\Gamma_{2.1.2}\right)\backslash N_{2.1.2}=\varnothing$,
hence the subcase resolves positively.
\item [Case 2.1.3.] $\psi_{2.1.3}=\psi_{u.y^{-1}}^{3}$, $u\mapsto uy^{-1}$,
$N_{2.1.3}=\{y.u^{-1},y^{-1}.y,y^{-1}.u^{-1},u.y\}$ and $\Gamma_{2.1.3}=\Gamma\left(\left\langle uy^{-1}uuy^{-1}y^{-1}\right\rangle \right)\twoheadrightarrow\Delta_{2.1.3}$.
We notice that $\gamma\colon\left(U_{5},N_{5}\right)\to\left(U_{2.1.3},N_{2.1.3}\right)$
defined by $u\mapsto u^{-1},v\mapsto y$ satisfies $\mathcal{F}_{\gamma}\left(\Gamma_{5}\to\Delta_{5}\right)$
is conjugate to $\Gamma_{2.1.3}\to\Delta_{2.1.3}$, which implies
that this subproblem is contained in Case 5.
\item [Case 2.1.4.] $\psi_{2.1.4}=\psi_{u.y^{-1}}^{4}$, $y\mapsto u^{-1}y$
and $N_{2.1.4}=\{y.u^{-1},u.y,u.u^{-1},u^{-1}.y^{-1}\}$; $N_{2.1.4}=N_{2.1}$
and $\Gamma_{2.1.4}=\Gamma_{2.1},\Delta_{2.1.4}=\Delta_{2.1}$, so
this is the same problem as Case 2.1.
\item [Case 3.] Here $\Gamma_{3}=\Gamma\left(\left\langle xuux^{-1}u\right\rangle \right)\twoheadrightarrow\Delta_{3}$
and $W\left(\Gamma_{3}\right)\backslash N_{3}=\{x^{-1}.u^{-1},u.x^{-1}\}$
We notice that $u.u^{-1}\in N_{2}$ so we can split using the triangle
rule.
\item [Case 3.1.] We have two cases $\psi_{3.1}=\psi_{u^{-1}.x^{-1}}^{1}$,
$N_{3.1}=\{x.u^{-1},u.x,u.u^{-1},u.x^{-1}\}$ and $\psi_{3.1'}=\psi_{u.x^{-1}}^{1}$,
$N_{3.1'}=\{x.u^{-1},u.x,u.u^{-1},u^{-1}.x^{-1}\}$. These cases are
equivalent, as $\gamma\colon\left(U_{3},N_{3.1}\right)\to\left(U_{3},N_{3.1'}\right)$
defined by $u\mapsto u^{-1},x\mapsto x$ satisfies $\mathcal{F}_{\gamma}\left(\Gamma_{3.1}\to\Delta_{3.1}\right)=\Gamma_{3.1'}\to\Delta_{3.1'}$.
Finally $\Gamma_{3.1}=\Gamma_{3},\Delta_{3.1}=\Delta_{3}$ and $W\left(\Gamma_{3.1}\right)\backslash N_{3.1}=\{u^{-1}.x^{-1}\}$.
We split by $u^{-1}.x^{-1}$.
\item [Case 3.1.1.] $\psi_{3.1.1}=\psi_{u^{-1}.x^{-1}}^{1}$ and $N_{3.1.1}=\{x.u^{-1},u.x,u.u^{-1},u.x^{-1},x^{-1}.u^{-1}\}$;
$\Gamma_{3.1.1}=\Gamma_{3.1},\Delta_{3.1.1}=\Delta_{3.1}$ and $W\left(\Gamma_{3.1.1}\right)\backslash N_{3.1.1}=\varnothing$
hence the subcase resolves positively.
\item [Case 3.1.2.] $\psi_{3.1.2}=\psi_{u^{-1}.x^{-1}}^{2}$, $x\mapsto t^{-1}x,\,u\mapsto t^{-1}u$
and $N_{3.1.2}=\{x.t,u.x,u.t,t^{-1}.x^{-1},\allowbreak t^{-1}.u^{-1},x^{-1}.u^{-1}\}$;
$\Gamma_{3.1.2}=\Gamma\left(\left\langle t^{-1}xt^{-1}ut^{-1}ux^{-1}u\right\rangle \right)\twoheadrightarrow\Delta_{3.1.2}$
and $W\left(\Gamma_{3.1.2}\right)\backslash N_{3.1.2}=\varnothing$,
hence the subcase resolves positively.
\item [Case 3.1.3.] $\psi_{3.1.3}=\psi_{u^{-1}.x^{-1}}^{3}$, $x\mapsto ux$
and $N_{3.1.3}=\{x.u^{-1},u.x,u.u^{-1},u.x^{-1}\}$; $N_{3.1.3}=N_{3.1}$
and $\Gamma_{3.1.3}=\Gamma_{3.1},\Delta_{3.1.1}=\Delta_{3.1}$. This
is the same problem as Case 3.1.
\item [Case 3.1.4.] $\psi_{3.1.4}=\psi_{u^{-1}.x^{-1}}^{4}$, $u\mapsto xu$
and $N_{3.1.4}=\{x.x^{-1},u.x,u.x^{-1},x.u^{-1}\}$; $\Gamma_{3.1.4}=\Gamma\left(\left\langle xxuxuu\right\rangle \right)\twoheadrightarrow\Delta_{3.1.4}$.
We notice that $\gamma\colon\left(U_{5},N_{5}\right)\to\left(U_{3.1.4},N_{3.1.4}\right)$
defined by $u\mapsto x,v\mapsto u$ satisfies that $\mathcal{F}_{\gamma}\left(\Gamma_{5}\to\Delta_{5}\right)$
is conjugate to $\Gamma_{3.1.4}\to\Delta_{3.1.4}$, hence this is
contained in Case 5.
\item [Case 4.] Here $\Gamma_{4}=\left(\left\langle xuux^{-1}yuy^{-1}\right\rangle \right)\twoheadrightarrow\Delta_{4}$
and $W\left(\Gamma_{4}\right)\backslash N_{4}=\{x^{-1}.y^{-1}\}$.
We split by $x^{-1}.y^{-1}\in W\left(\Gamma_{4}\right)\backslash N_{4}$.
\item [Case 4.1.] $\psi_{4.1}=\psi_{v.v^{-1}}^{1}$, $N_{4.1}=\{x.u^{-1},u.x,u.u^{-1},y.u^{-1},u.y,y.x,x^{-1}.y^{-1}\}$;
$\Gamma_{4.1}=\Gamma_{4},\Delta_{4.1}=\Delta_{4}$ and $W\left(\Gamma_{4.1}\right)\backslash N_{4.1}=\varnothing$,
hence the subcase resolves positively.
\item [Case 4.2.] $\psi_{4.2}=\psi_{x^{-1}.y^{-1}}^{2}$, $x\mapsto t^{-1}x$,
$y\mapsto t^{-1}y$ and $N_{4.2}=\{x.u^{-1},u.x,u.u^{-1},y.u^{-1}y.u,y.x,\allowbreak x^{-1}.t^{-1},y^{-1}.t^{-1},y^{-1}.x^{-1}\}$;
after conjugation $\Gamma_{4.2}=\Gamma_{4}$, $\Delta_{4.2}=\Delta_{4}$
and $W\left(\Gamma_{4.2}\right)\backslash N_{4.2}=\varnothing$, hence
the subcase resolves positively.
\item [Case 4.3.] $\psi_{4.3}=\psi_{x^{-1}.y^{-1}}^{3}$, $x\mapsto yx$
and $N_{4.3}=\{x.u^{-1},u.x,u.u^{-1},y.u^{-1},u.y,y.x,y.x^{-1}\}$.
We notice that $\gamma\colon\left(U_{3},N_{3}\right)\to\left(U_{4.3},N_{4.3}\right)$
defined by $u\mapsto u,x\mapsto x$ satisfies $\mathcal{F}_{\gamma}\left(\Gamma_{3}\to\Delta_{3}\right)=\Gamma_{4.3}\to\Delta_{4.3}$,
hence this problem is contained in Case 3.
\item [Case 4.4.] $\psi_{4.4}=\psi_{x^{-1}.y^{-1}}^{4}$, $y\mapsto xy$
and $N_{4.4}=\{x.u^{-1},u.x,u.u^{-1},y.u^{-1},u.y,y.x,x.y^{-1}\}$.
We notice that $\gamma\colon\left(U_{2},N_{2}\right)\to\left(U_{4.4},N_{4.4}\right)$
defined by $u\mapsto u,y\mapsto y$ satisfies $\mathcal{F}_{\gamma}\left(\Gamma_{2}\to\Delta_{2}\right)=\Gamma_{4.4}\to\Delta_{4.4}$,
hence this problem is contained in Case 2.
\item [Case 6.] Here $\Gamma_{6}=\left(\left\langle uvuvyvuy^{-1}\right\rangle \right)\twoheadrightarrow\Delta_{6}$
and $W\left(\Gamma_{6}\right)\backslash N_{6}=\{y^{-1}.u^{-1},v.y^{-1}\}$.
We notice that $v.u^{-1}\in N_{2}$ so we can split using the triangle
rule.
\item [Case 6.1.] We have two cases $\psi_{6.1}=\psi_{u^{-1}.y^{-1}}^{1}$,
$N_{6.1}=\{v.u^{-1},u.v^{-1},y.v^{-1},u.y,y^{-1}.u^{-1}\}$ and $\psi_{6.1'}=\psi_{v.y^{-1}}^{1}$,
$N_{6.1'}=\{v.u^{-1},u.v^{-1},y.v^{-1},u.y,v.y^{-1}\}$. These cases
are equivalent, as $\gamma\colon\left(U_{6},N_{6.1}\right)\to\left(U_{6},N_{6.1'}\right)$
defined by $u\mapsto v^{-1},v\mapsto u^{-1}$ satisfies $\mathcal{F}_{\gamma}\left(\Gamma_{6.1}\to\Delta_{6.1}\right)=\Gamma_{6.1'}\to\Delta_{6.1'}$.
$\Gamma_{6.1}=\Gamma_{6},\Delta_{6.1}=\Delta_{6}$ and $W\left(\Gamma_{6.1}\right)\backslash N_{6.1}=\{v.y^{-1}\}$.
We split by $v.y^{-1}$.
\item [Case 6.1.1.] $\psi_{6.1.1}=\psi_{v.y^{-1}}^{1}$ and $N_{6.1.1}=\{v.u^{-1},u.v^{-1},y.v^{-1},u.y,y^{-1}.u^{-1},v.y^{-1}\}$;
$\Gamma_{6.1.1}=\Gamma_{6},\Delta_{6.1.1}=\Gamma_{6}$ and $W\left(\Gamma_{6.1.1}\right)\backslash N_{6.1.1}=\varnothing$,
hence the subcase resolves positively.
\item [Case 6.1.2.] $\psi_{6.1.2}=\psi_{v.y^{-1}}^{2}$, $v\mapsto vt,\,y\mapsto t^{-1}y$
and $N_{6.1.2}=\{t.u^{-1},u.v^{-1},y.v^{-1},u.y,v.t^{-1},\allowbreak t^{-1}.y^{-1}\}$;
$\Gamma_{6.1.2}=\Gamma\left(\left\langle uvtuvyvtuy^{-1}t\right\rangle \right)\twoheadrightarrow\Delta_{6.1.2}$
and $W\left(\Gamma_{6.1.2}\right)\backslash N_{6.1.2}=\varnothing$,
hence the subcase resolves positively.
\item [Case 6.1.3.] $\psi_{6.1.3}=\psi_{v.y^{-1}}^{3}$, $v\mapsto vy^{-1}$,
$N_{6.1.3}=\{y^{-1}.u^{-1},u.v^{-1},y.v^{-1},u.y,y^{-1}.u^{-1},v.y\}$
and $\Gamma_{6.1.3}=\Gamma\left(\left\langle y^{-1}uvy^{-1}uvvy^{-1}u\right\rangle \right)$.
We notice that $\gamma\colon\left(U_{5},N_{5}\right)\to\left(U_{6.1.3},N_{6.1.3}\right)$
defined by $u\mapsto y^{-1}u,v\mapsto v$ satisfies $\mathcal{F}_{\gamma}\left(\Gamma_{5}\to\Delta_{5}\right)$
is conjugate to $\Gamma_{6.1.3}\to\Delta_{6.1.3}$ therefore this
problem is contained in Case 5.
\item [Case 6.1.4.] $\psi_{6.1.4}=\psi_{v.y^{-1}}^{4}$, $y\mapsto v^{-1}y$,
$N_{6.1.4}=\{v.u^{-1},u.v^{-1},y.v^{-1},u.y,v^{-1}.y^{-1}\}$ and
$\Gamma_{6.1.4}\left(\left\langle vuvuyvuy^{-1}\right\rangle \right)$.
We notice that $\gamma\colon\left(U_{2},N_{2}\right)\to\left(U_{6.1.4},N_{6.1.4}\right)$
defined by $u\mapsto vu$ satisfies $\mathcal{F}_{\gamma}\left(\Gamma_{2}\to\Delta_{2}\right)=\Gamma_{6.1.4}\to\Delta_{6.1.4}$,
hence this problem is contained in Case 2.
\item [Case 7.] Here $\Gamma_{7}=\left(\left\langle xuvuvx^{-1}vu\right\rangle \right)\twoheadrightarrow\Delta_{7}$
and $W\left(\Gamma_{7}\right)\backslash N_{7}=\{x^{-1}.v^{-1},u.x^{-1}\}$.
We notice that $u.v^{-1}\in N_{2}$ so we can split using the triangle
rule.
\item [Case 7.1] We have two cases $\psi_{7.1}=\psi_{u.x^{-1}}^{1}$, $N_{7.1}=\{v.u^{-1},u.v^{-1},x.u^{-1},v.x,x^{-1}.u\}$
and $\psi_{7.1'}=\psi_{x^{-1}.v}^{1}$, , $N_{7.1'}=\{v.u^{-1},u.v^{-1},x.u^{-1},v.x,x^{-1}.v^{-1}\}$.
These cases are equivalent, as $\gamma\colon\left(U_{7},N_{7.1}\right)\to\left(U_{7},N_{7.1'}\right)$
defined by $u\mapsto v^{-1},v\mapsto u^{-1}$ satisfies $\mathcal{F}_{\gamma}\left(\Gamma_{7.1}\to\Delta_{7.1}\right)=\Gamma_{7.1'}\to\Delta_{7.1'}$;
$\Gamma_{7.1}=\Gamma_{7},\Delta_{7.1}=\Gamma_{7}$ and $W\left(\Gamma_{7.1}\right)\backslash N_{7.1}=\{x^{-1}.v^{-1}\}$.
We split by $x^{-1}.v^{-1}$.
\item [Case 7.1.1.] $\psi_{7.1.1}=\psi_{x^{-1}.v^{-1}}^{1}$ and $N_{7.1.1}=\{v.u^{-1},u.v^{-1},x.u^{-1},v.x,x^{-1}.u,x^{-1}.v^{-1}\}$;
$\Gamma_{7.1.1}=\Gamma_{7},\Delta_{7.1.1}=\Delta_{7}$ and $W\left(\Gamma_{7.1.1}\right)\backslash N_{7.1.1}=\varnothing$,
hence the subcase resolves positively.
\item [Case 7.1.2.] $\psi_{7.1.2}=\psi_{x^{-1}.v^{-1}}^{2}$, $x\mapsto t^{-1}x,\,v\mapsto t^{-1}v$
and $N_{7.1.2}=\{v.u^{-1},u.t,x.u^{-1},v.x,t^{-1}.x^{-1},t^{-1}.v^{-1},x^{-1}.v^{-1}\}$;
$\Gamma_{7.1.2}=\Gamma\left(\left\langle ut^{-1}vut^{-1}vx^{-1}vut^{-1}xu^{-1}\right\rangle \right)\twoheadrightarrow\Delta_{7.1.2}$
and $W\left(\Gamma_{7.1.2}\right)\backslash N_{7.1.2}=\varnothing$,
hence the subcase resolves positively.
\item [Case 7.1.3] $\psi_{7.1.3}=\psi_{x^{-1}.v^{-1}}^{4}$, $x\mapsto vx$,
$N_{7.1.3}=\{v.u^{-1},u.v^{-1},x.u^{-1},v.x,v.x^{-1}\}$ and $\Gamma_{7.1.3}\left(\left\langle xuvuvx^{-1}uv\right\rangle \right)$.
We notice that $\gamma\colon\left(U_{3},N_{3}\right)\to\left(U_{7.1.3},N_{7.1.3}\right)$
defined by $u\mapsto uv,x\mapsto x$ satisfies $\mathcal{F}_{\gamma}\left(\Gamma_{3}\to\Delta_{3}\right)=\Gamma_{7.1.3}\to\Delta_{7.1.3}$,
hence this problem is contained in Case 3.
\item [Case 7.1.4] $\psi_{7.1.3}=\psi_{x^{-1}.v^{-1}}^{4}$, $v\mapsto xv$,
$N_{7.1.4}=\{v.u^{-1},u.x^{-1},x.u^{-1},v.x,x.v^{-1}\}$ $\Gamma_{7.1.4}=\left(\left\langle uxvuxvvux\right\rangle \right)$.
We notice that $\gamma\colon\left(U_{5},N_{5}\right)\to\left(U_{7.1.4},N_{7.1.4}\right)$
defined by $u\mapsto ux,v\mapsto v$ satisfies $\mathcal{F}_{\gamma}\left(\Gamma_{5}\to\Delta_{5}\right)=\Gamma_{7.1.4}\to\Delta_{7.1.4}$,
hence this problem is contained in Case 5.
\item [Case 8.]Here $\Gamma_{8}\left(xuvuvx^{-1}yvuy^{-1}\right)\twoheadrightarrow\Delta_{8}$
and$W\left(\Gamma_{8}\right)\backslash N_{8}=\{x^{-1}.y^{-1}\}$.
We split by $x^{-1}.y^{-1}$.
\item [Case 8.1.] $\psi_{8.1}=\psi_{x^{-1}.y^{-1}}^{1}$, $N_{8.1}=\{v.u^{-1},u.v^{-1},x.u^{-1},v.x,\,y.v^{-1},u.y,x^{-1}.y^{-1}\}$;
$\Gamma_{8.1}=\Gamma_{8},\Delta_{8.1}=\Delta_{8}$ and $W\left(\Gamma_{8.1}\right)\backslash N_{8.1}=\varnothing$
hence the subcase resolves positively.
\item [Case 8.2.] $\psi_{8.2}=\psi_{x^{-1}.y^{-1}}^{2}$, $x\mapsto t^{-1}x,\,y\mapsto t^{-1}y$
and $N_{8.2}=\{v.u^{-1},u.v^{-1},x.u^{-1},v.x,\,y.v^{-1},u.y,\allowbreak t^{-1}.y^{-1},t^{-1}.x^{-1},x^{-1}.y^{-1}\}$;
$\Gamma_{8.2},\Delta_{8.2}$ are conjugate to $\Gamma_{8},\Delta_{8}$
and $W\left(\Gamma_{8.2}\right)\backslash N_{8.2}=\varnothing$ hence
the subcase resolves positively.
\item [Case 8.3.] $\psi_{8.3}=\psi_{x^{-1}.y^{-1}}^{3}$, $x\mapsto yx$
and $N_{8.3}=\{v.u^{-1},u.v^{-1},x.u^{-1},v.x,\,y.v^{-1},u.y,y.x^{-1}\}$.
We notice that $\gamma\colon\left(U_{7},N_{7}\right)\to\left(U_{8.3},N_{8.3}\right)$
defined by $u\mapsto u,x\mapsto x,v\mapsto v$ satisfies $\mathcal{F}_{\gamma}\left(\Gamma_{7}\to\Delta_{7}\right)=\Gamma_{8.3}\to\Delta_{8.3}$,
hence this problem is contained in Case 7.
\item [Case 8.4.] $\psi_{8.4}=\psi_{x^{-1}.y^{-1}}^{4}$, $y\mapsto xy$
and $N_{8.4}=\{v.u^{-1},u.v^{-1},x.u^{-1},v.x,\,y.v^{-1},u.y,x.y^{-1}\}$.
We notice that $\gamma\colon\left(U_{6},N_{6}\right)\to\left(U_{8.4},N_{8.4}\right)$
defined by $u\mapsto u,y\mapsto y,v\mapsto v$ satisfies $\mathcal{F}_{\gamma}\left(\Gamma_{6}\to\Delta_{6}\right)=\Gamma_{8.4}\to\Delta_{8.4}$,
hence this problem is contained in Case 6.
\item [Case 8.5.] $\psi_{8.5}=\psi_{x^{-1}.y^{-1}}^{5}$, $y\mapsto x$
and $N_{8.4}=\{v.u^{-1},u.v^{-1},x.u^{-1},v.x,\,x.v^{-1},u.x\}$.
We notice that $\gamma\colon\left(U_{5},N_{5}\right)\to\left(U_{8.4},N_{8.4}\right)$
defined by $u\mapsto u,v\mapsto v$ satisfies $\mathcal{F}_{\gamma}\left(\Gamma_{5}\to\Delta_{5}\right)=\Gamma_{8.5}\to\Delta_{8.5}$,
hence this problem is contained in Case 5.
\item [Case 1.] We notice that $\Gamma_{1}\twoheadrightarrow\Delta_{1}$
and that $W\left(\Gamma_{1}\right)\backslash N_{1}=\varnothing$,
hence the case resolves positively.
\end{casenv}
\begin{figure}[h]
\begin{centering}
\includegraphics[scale=0.8]{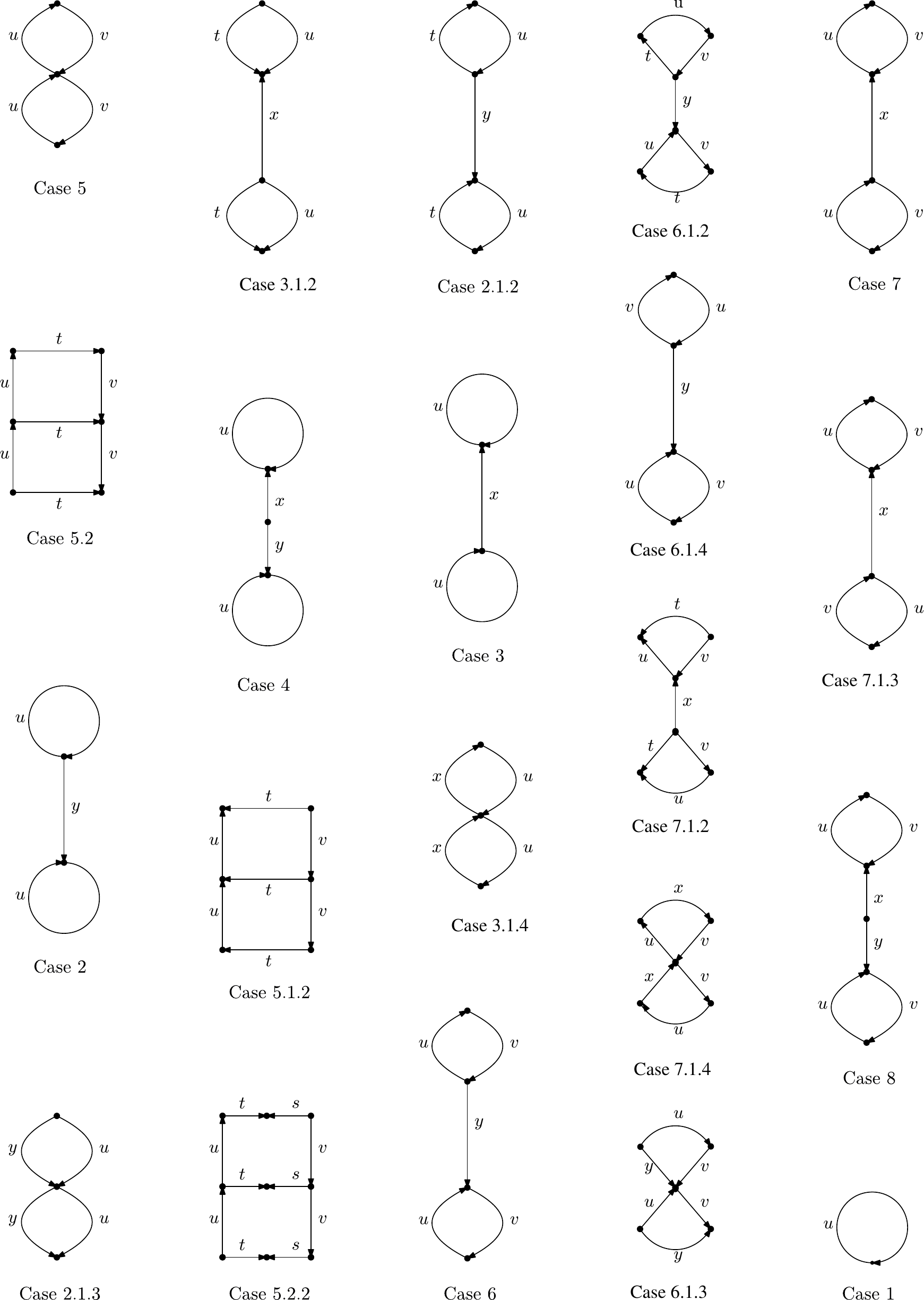}\bigskip{}
\par\end{centering}
\caption{\label{fig:The-graphs-}The graphs $\Delta_{i.k.j}=\protect\core\mathcal{F}_{\psi_{i.k.j}}\left(\Delta_{i,k}\right)$
arising in the analysis of the counterexample.}
\end{figure}

\section{\label{sec:TBD-finiteness}Stencil-finiteness}

Let $\mathscr{P}=\{\Gamma\to\Delta\,|\,\left(U,N_{U}\right)\}$ be
a surjectivity problem. In this section we define and examine a property
of the graph $\Gamma$, that is necessary for the recursive splitting
of $\mathscr{P}$ (as described in §\ref{sec:General-setting}) to
end in a finite set of equivalent ambiguous cases. Examining Example
\ref{exa:comutator} we see that all the graphs $\left\{ \core\mathcal{F}_{\varphi}\left(\Gamma\left(xyx^{-1}y^{-1}\right)\right)|\varphi\in\hom\left(\left(U,N_{U}\right),\left(X,W_{X}\right)\right)\right\} $
are contained in two stencil spaces $\left\{ \mathcal{F}_{\varphi}\left(\Gamma\left(xyx^{-1}y^{-1}\right)\right)|\varphi\in\hom\left(\left(U_{1},N_{U_{1}}\right),\left(X,W_{X}\right)\right)\right\} $
and $\left\{ \mathcal{F}_{\varphi}\left(\Gamma\left(xtyx^{-1}t^{-1}y^{-1}\right)\right)|\varphi\in\hom\left(\left(U_{2},N_{U_{2}}\right),\left(X,W_{X}\right)\right)\right\} $.
This is a reformulation of a classic result of Wicks \cite{MR0142610}
about the commutator, as we explain below. A closer examination of
the computation in §\ref{sec:The-counterexample} shows that there
are five stencil spaces (up to conjugation) that include all other
stencil spaces: Cases 1, 2.1, 5.1.1, 5.1.2, and 6.1. The set of graphs
$\left\{ \core\mathcal{F}_{\varphi}\Gamma\left(\left\langle bbaba^{-1}\right\rangle \right)|\varphi\in\hom\left(\left(\left\{ a,b\right\} ,\emptyset\right),\left(X,W_{X}\right)\right)\right\} $
is the union of the stencil spaces $\bigcup_{i\in I}\left\{ \mathcal{F}_{\varphi}\left(\Gamma_{i}\right)|\varphi\in\hom\left(\left(U_{i},N_{U_{i}}\right),\left(X,W_{X}\right)\right)\right\} $
with $I=\{1,\,2.1,\,5.1.1,\,5.1.2\,,6.1\}$ the set of indices of
stencil space cases. This fact, formulated differently, appears in
\cite{MR285589} (where Meskin attributes it to unpublished work by
Newman). In fact, he shows more generally that this is true for all
Baumslag-Solitar relators $b^{l}ab^{m}a^{-1}$ $\left(l,m\in\mathbb{Z}\right)$.
Technically this generalizes Wicks' result, as the commutator has
$l=1,m=-1$. This leads us to define the notion of stencil-finiteness
below. Even though Meskin and Wicks did not define this notion, each
of them showed that some word satisfies it. They were interested in
words $w$ with this property because it implies an easy algorithm
to determine whether there is a solution to the equation $w\left(\oton yn\right)=a$
with $a\in F_{X}$ and unknowns $\oton yn\in F_{Y}$. First we define
the following poset:
\begin{defn}
Let $X$ be a countably infinite set, and define an equivalence relation
between folded $X$-labeled graphs: $\Gamma\sim\Delta$ iff there
exists non-degenerate homomorphisms $\varphi,\psi\colon F_{X}\to F_{X}$
such that\ $\left(\varphi,\Gamma\right)$ and $\left(\psi,\Delta\right)$
are stencils and $\mathcal{F}_{\varphi}\left(\Gamma\right)=\Delta$,
$\mathcal{F}_{\psi}\left(\Delta\right)=\Gamma$. On the set of equivalence
classes we define $[\Gamma]\leq[\Delta]$ iff there exists a non-degenerate
homomorphism $\varphi\colon F_{X}\to F_{X}$ such that\ $\left(\varphi,\Delta\right)$
is a stencil and $\mathcal{F}_{\varphi}\left(\Delta\right)=\Gamma$.
\end{defn}

\begin{defn}
\label{prob:Let--be}We say that a folded $Y$-labeled graph $\Delta$
has \textit{stencil-finiteness} if the set $\left\{ \left[\core\left(\mathcal{F}_{\varphi}\left(\Delta\right)\right)\right]|\varphi\in\hom\left(\left(Y,\varnothing\right),(X,W_{X})\right)\right\} $
has a finite number of maximal elements.
\end{defn}

This property can be translated to a property of words, subgroups
and homomorphisms of free groups by taking the corresponding graphs.
Not all words have this property: Using the methods described in this
paper one can show that the word $x^{3}y^{2}$ does not have this
property. This is a lengthy computation and not the subject matter
of the paper, and is therefore omitted. Is stencil finiteness a sufficient
condition for the process ending? We do not know the answer. A priori,
the graph $\Gamma$ may have stencil finiteness while $\Delta$ doesn't,
and while the graphs $\core\mathcal{F}_{\varphi}\left(\Gamma\right)$
could repeat themselves, the graphs $\core\mathcal{F}_{\varphi}\left(\Delta\right)$
can be distinct and thus lead to infinite distinct ambiguous cases.
We do not have an explicit example for this mainly because there aren't
many graphs known to have stencil finiteness. The words known to have
stencil finiteness are $b^{l}ab^{m}a^{-1}$ with $l,m\in\mathbb{Z}$
(this includes the commutator), primitive words. One can show that
powers of a word with this property\textcolor{red}{{} }have it as well.

We end with some open questions:
\begin{enumerate}
\item Are there any other words with this property? Are there core graphs
of subgroups with rank $\geq2$ that have this property?
\item Is a surjectivity problem $\mathscr{P}=\left\{ \Gamma\to\Delta|\left(U,N_{U}\right)\right\} $
such that $\Gamma$ has stencil finiteness decidable?
\item Is stencil finiteness a decidable property?
\item Is there a purely algebraic interpretation of the property of an extension
being onto on every base?
\end{enumerate}
\bibliographystyle{plain}
\bibliography{mybib}

\end{document}